\pdfoutput=1
\documentclass[leqno]{article} 
\usepackage[vcentering]{geometry}
\usepackage{graphicx,xcolor}
\usepackage{tikz}
\usetikzlibrary{cd}
\usepackage[backend=biber,sorting=nyt]{biblatex}
\usepackage{amsmath,amssymb}
\usepackage{rsfso}
\allowdisplaybreaks[1]
\numberwithin{equation}{section}
\usepackage[%
    bookmarksnumbered=true%
]{hyperref}
\usepackage{amsthm}
\theoremstyle{plain}
\newtheorem{Th}{Theorem}[section]
\newtheorem{Prop}[Th]{Proposition}
\newtheorem{Lem}[Th]{Lemma}
\newtheorem{Cor}[Th]{Corollary}

\theoremstyle{definition}
\newtheorem{Def}[Th]{Definition}
\newtheorem{Exm}[Th]{Example}
\newtheorem{Rem}[Th]{Remark}

\newcommand{\lrbrack}{\mathopen{\rbrack}}
\newcommand{\rlbrack}{\mathclose{\lbrack}}
\newcommand{\Par}{\mathop{\mathrm{Par}}}
\newcommand{\cadlag}{c\`{a}dl\`{a}g}
\usepackage{newtxtext}
\usepackage[varbb]{newtxmath}

\renewcommand{\thefootnote}{\fnsymbol{footnote}}
\begin{document}
\title{
    It\^{o}--F\"{o}llmer Calculus in Banach Spaces II: \\
    Transformations of Quadratic Variations
}
\author{Yuki Hirai}
\date{}
\maketitle

\begin{abstract}
    In this paper, we study properties of quadratic variations 
    of c\`{a}dl\`{a}g paths within the framework of the It\^{o}--F\"{o}llmer calculus
    in Banach spaces.
    We prove a $C^1$-type transformation formula for
    quadratic variations. 
    We also investigate relations between tensor and scalar
    quadratic variations.
\end{abstract}

\renewcommand{\thefootnote}{\arabic{footnote}}

\tableofcontents

\section{Introduction}
The It\^{o}--F\"{o}llmer calculus,
which originated in F\"{o}llmer~\cite{Foellmer_1981},
is a deterministic counterpart to classical It\^{o}'s stochastic calculus.
Recently, the It\^{o}--F\"{o}llmer calculus has seen increasing developments,
receiving much attention from the viewpoint of its financial applications
(see, e.g., Sondermann~\cite{Sondermann_2006},
Schied~\cite{Schied_2014,Schied_2016},
Mishura and Schied~\cite{Mishura_Schied_2016},
Schied, Speiser, and Voloshchenko~\cite{Schied_Speiser_Voloshchenko_2018},
Chiu and Cont~\cite{Chiu_Cont_2018},
Hirai~\cite{Hirai_2019},
and Cont and Perkowski~\cite{Cont_Perkowski_2019}).
In addition, in the spirit of functional It\^{o} calculus
(Dupire~\cite{Dupire_2009}
and Cont and Fourni\'{e}~\cite{Cont_Fournie_2010b,Cont_Fournie_2013}),
It\^{o}--F\"{o}llmer calculus has been extended
to path-dependent functionals;
see Cont and Fourni\'{e}~\cite{Cont_Fournie_2010}
and Ananova and Cont~\cite{Ananova_Cont_2017}.
Extensions of the It\^{o}--F\"{o}llmer calculus
in terms of local times have been investigated in
Davis, Ob{\l}\'{o}j, and Raval~\cite{Davis_Obloj_Raval_2014},
Davis, Ob{\l}\'{o}j, and Siorpaes~\cite{Davis_Obloj_Siorpaes_2018},
Kim~\cite{Kim_2022},
{\L}ochowski et al.~\cite{Lochowski_Obloj_Promel_Siorpaes_2021b},
and Hirai~\cite{Hirai_2016,Hirai_2021}.

Stochastic integration in infinite dimensions 
appears naturally 
when handling stochastic partial differential equations (SPDEs).
SPDEs have been applied to various areas,
including finance and physics.
One can refer to, for example, Da Prato and Zabczyk~\cite{DaPrato_Zabczyk_2014} and 
Peszat and Zabczyk~\cite{Peszat_Zabczyk_2007} for SPDEs and their applications. 
To the best of our knowledge, however,
the It\^{o}--F\"{o}llmer calculus in infinite dimensions
has not yet been sufficiently studied.
In this series of articles, we aim to develop a theory of 
It\^{o}--F\"{o}llmer calculus in Banach spaces.
In our previous paper~\cite{Hirai_2022a},
we proved the It\^{o} formula for Banach space-valued 
paths with suitably defined quadratic variations.
We also relaxed the assumption on the sequence of
partitions along which we consider the quadratic variation.

In this paper, we study more detailed properties of 
quadratic variations,
focusing on their transformations and representations.
In particular, we first study the relations between 
tensor, scalar, and other quadratic variations.
We provide some linear transformation formulae,
one of which implies that the existence of
projective tensor quadratic variation implies
the existence of all quadratic variations with respect to bounded bilinear maps.
We then prove the trace representation formula 
of scalar quadratic variations by tensor one
in a Hilbert space setting.
We also show an integral representation formula 
of quadratic variation in a Banach space with 
the Radon--Nikodym property (RNP).
Both trace and integral representation 
formulae of quadratic variations are 
interpreted as a pathwise version of 
classical results (see, for example, Metivier~\cite{Metivier_1982})
in stochastic calculus in Hilbert spaces.
Next, we show $C^1$-type transformation formulae for quadratic variations.
Roughly speaking, these results 
claim that the path $t \mapsto f(t,X_t)$
has quadratic variation whenever $X$ has quadratic variation
and $f$ has $C^1$-smoothness in $x$ and finite variation in $t$.
These can be interpreted as a method of generating
new {\cadlag} paths having quadratic variation
from the given path $X$.
Furthermore, we investigate quadratic 
variations of a path $\int_0^{\cdot} Df(X_{s-}) \mathrm{d}X_s$
defined by the It\^{o}--F\"{o}llmer integral.
As an application, we obtain a rough-FV decomposition formula 
for a path of the form $f(X)$.

Stochastic integration theories in infinite dimensions
have been investigated in many studies (see, e.g., Kunita~\cite{Kunita_1970},
Metivier~\cite{Metivier_1972},
Pellaumail~\cite{Pellaumail_1973},
Yor~\cite{Yor_1974},
Gravereaux and Pellaumail~\cite{Gravereaux_Pellaumail_1974},
Metivier and Pistone~\cite{Metivier_Pistone_1975},
Meyer~\cite{Meyer_1977a},
Metivier and Pellaumail~\cite{Metivier_Pellaumail_1980b},
Gy\"{o}ngy and Krylov~\cite{Gyongy_Krylov_1980,Gyongy_Krylov_1982},
Gy\"{o}ngy~\cite{Gyongy_1982},
Metivier~\cite{Metivier_1982},
Pratelli~\cite{Pratelli_1988},
Brooks and Dinculeanu~\cite{Brooks_Dinculeanu_1990},
Dinculeanu~\cite{Dinculeanu_2000}, and
van Neerven, Veraar, and Weis~\cite{vanNeerven_Veraar_Weis_2007a,vanNeerven_Veraar_Weis_2015}).
Our method can be interpreted as a deterministic analogy
of some of these classical theories.
Note that Di Girolami, Fabbri, and Russo~\cite{DiGirolami_Fabbri_Russo_2014}
and Di Girolami and Russo~\cite{DiGirolami_Russo_2014}
treat quadratic covariations and stochastic integrations in Banach spaces
with a different approach, called stochastic calculus via regularization.
These are considered to be regularization counterparts
of F\"{o}llmer's discretization approach.
See the survey by Russo and Vallois~\cite{Russo_Vallois_2007}
for details of stochastic calculus via regularization.
We also refer to \v{C}oupek and Garrido-Atienza~\cite{Coupek_GarridoAtienza_2021}
for linear equations in a Hilbert space driven by irregular scalar noise,
which works within the framework of F\"{o}llmer's calculus.

Before describing the results of this paper,
we give a brief outline of the
It\^{o}--F\"{o}llmer calculus in Banach spaces
developed in our previous study~\cite{Hirai_2022a}.
Let $(\pi_n)_{n \in \mathbb{N}}$ be a sequence of partitions of
$\mathbb{R}_{\geq 0}$ and
let $X \colon \mathbb{R}_{\geq 0} \to E$
be a c\`{a}dl\`{a}g path in a Banach space $E$.
Moreover, let $B \colon E \times E \to F$ be a
bounded bilinear map into another Banach space $F$.
We say that $X$ has strong (resp. weak) $B$-quadratic variation
along $(\pi_n)$ if there is a c\`{a}dl\`{a}g path
$Q_B(X,X) \colon \mathbb{R}_{\geq 0} \to F$
of finite variation satisfying the following conditions:
\begin{enumerate}
    \item the sequence $\sum_{\lrbrack r,s \rbrack \in \pi_n} B(X_{s \wedge t} - X_{r \wedge t},X_{s \wedge t} - X_{r \wedge t})$ 
        converges to $Q_B(X,X)_t$
        in the norm (resp. weak) topology for all $t \geq 0$;
    \item the equation $\Delta [X,Y]_t = B(\Delta X_t,\Delta X_t)$ holds for all $t \geq 0$.
\end{enumerate}
A typical example of the bilinear map $B$ in the definition above
is the canonical bilinear map into a tensor product of Banach spaces.
Let $\alpha$ be a reasonable crossnorm on the algebraic
tensor product $E \otimes E$, and 
let $E \widehat{\otimes}_{\alpha} E$ be the
completion of $E \otimes E$ with respect to $\alpha$.
Then the strong (resp. weak) quadratic variation $Q_{\otimes}(X,X)$,
where $\otimes \colon E \times E \to E \widehat{\otimes}_{\alpha} E$
is the canonical bilinear map,
is called the strong (resp. weak) $\alpha$-tensor quadratic variation of $X$
and is denoted by ${^\alpha [X,X]}$.
Next, we define the scalar quadratic variation $Q(X)$ of $X$ along $(\pi_n)$
as a nonnegative increasing path satisfying the following conditions:
\begin{enumerate}
    \item the sequence $\sum_{\lrbrack r,s \rbrack \in \pi_n} \lVert X_{s \wedge t} - X_{r \wedge t} \rVert^2$ 
        converges to $Q(X)_t$ for all $t \geq 0$;
    \item the equation $\Delta Q(X)_t = \lVert \Delta X_t \rVert^2$ holds for all $t \geq 0$.
\end{enumerate}
Note that the scalar quadratic variation coincides
with the quadratic variation $Q_{\langle \phantom{x},\phantom{x} \rangle}$
if $E$ is a Hilbert space with inner product $\langle \phantom{x},\phantom{x} \rangle$. 
Moreover, we say that $X$ has finite 2-variation along $(\pi_n)$ if
\begin{equation*}
    V((\pi_n);X)_t
    =
    \sup_{n \in \mathbb{N}}
        \sum_{\lrbrack r,s \rbrack \in \pi_n} \lVert X_{s \wedge t} - X_{r \wedge t} \rVert^2
    < \infty
\end{equation*} 
for all $t \geq 0$.
Now suppose that $(\pi_n)$ satisfies certain nice conditions
associated with the path $X$,
namely condition~(C) and left-approximation (see Definition~\ref{2b}),
and $X$ has strong or weak $\alpha$-tensor quadratic variation
and finite 2-variation along $(\pi_n)$.
If a function $f \colon E \to G$ into a Banach space
has $C^2$ smoothness in an appropriate sense
(see Corollary~\ref{2fc} for the precise definition),
then the composite function $f \circ X$ satisfies
\begin{align*} 
    f(X_t) - f(X_0) 
    & =
    \int_0^t D f (X_{s-})\mathrm{d}X_s
    + \frac{1}{2} \int_0^t D^2 f(X_{s-})\mathrm{d}{^\alpha [X,X]}^{\mathrm{c}}_s
    + \sum_{0 < s \leq t} \left\{ \Delta f(X_{s}) - D f(X_{s-}) \Delta X_s \right\}. 
\end{align*}
The first integral on the right-hand side
is defined by
\begin{equation*}
    \lim_{n \to \infty} \sum_{\lrbrack u,v \rbrack \in \pi_n} 
        D f(X_{u})(X_{v \wedge t} - X_{u \wedge t})
    =
    \int_0^t D f (X_{s-})\mathrm{d}X_s,
\end{equation*}
where the limit exists in the strong or weak topology
corresponding to the convergence of the quadratic variation. 
We call $\int_0^t D f (X_{s-})\mathrm{d}X_s$
the It\^{o}--F\"{o}llmer integral along $(\pi_n)$.

The remainder of this paper is organized as follows.
In Section~2, we present the basic notation and terminology
used throughout the paper.
The remaining sections are divided into two parts.
In the first part, which consists of Sections~3 and 4,
we study relations between various quadratic variations
introduced in our previous paper~\cite{Hirai_2022a}.
Sections~5 and 6 form the second part,
where we focus on the formulation and the proof 
of our $C^1$-transformation formula
and its consequences. 

In Section~3, we provide some results related
to linear transformations of paths and quadratic variations.
It turns out that the projective tensor quadratic variation is the 
strongest notion in the sense of Proposition~\ref{3c}. 
We also show the trace representation formula of scalar quadratic variations
in Hilbert spaces (Theorem~\ref{3g}),
which claims that the scalar quadratic variation of 
a Hilbert space-valued path is
the trace of the projective tensor quadratic variation.
In Section~4, we show integral representation formulae of 
quadratic variations with respect to scalar quadratic variations.
Namely, if $X$ has scalar quadratic variation
and weak $B$-quadratic variation $Q_B$ in a Banach space $G$ with the RNP,
there is a $G$-valued locally $Q(X)$-integrable function $q_B$ such that 
\begin{equation*}
    Q_B(X,X)_t = \int_{\lrbrack 0,t \rbrack} q_B(s) \mathrm{d}Q(X)_s, \qquad \forall t \geq 0
\end{equation*}
(Theorem~\ref{4i}).
For this purpose, we also show the absolute continuity of 
$Q_B(X,X)$ with respect to the scalar quadratic variation
(Proposition~\ref{4d}).

The $C^1$-transformation formulae of quadratic variations
are given in Section~5.
Section~5.1 is the preliminary part of this section.
To formulate our $C^1$-transformation formulae, 
we first introduce some concepts
such as the variation of a family of {\cadlag} paths,
conditions on a sequence of partitions (Condition~(UC)),
and spaces of functions with certain continuity or differentiability.
We also provide auxiliary results 
regarding a {\cadlag} path in such a function space.
The main results of the second part,
$C^1$-transformation formulae for quadratic variations,
are given in Section~5.2
(Theorem~\ref{5.2b} and Corollaries~\ref{5.2kb} and \ref{5.2l}).
The statement of Corollary~\ref{5.2kb} is roughly as follows:
Let $E$ and $F$ be Banach spaces and $\alpha$ be a uniform crossnorm.
We consider two {\cadlag} paths
$X \colon \mathbb{R}_{\geq 0} \to E$
and $f \colon \mathbb{R}_{\geq 0} \to C^1_{\mathcal{K}}(E,F)$,
where $C^1_{\mathcal{K}}(E,F)$ is the space of G\^{a}teaux
differentiable functions from $E$ to $F$ satisfying additional conditions
(see Definition~\ref{5.1d}).
Assume that the family $(f(\,\cdot\, ,x); x \in K)$
has uniformly finite variation for each compact set $K \subset E$
and that $(\pi_n)$ satisfies (UC) and is a left-approximation sequence for $X$ and $f$.
If $X$ has strong (resp. weak) $\alpha$-tensor quadratic variation
and finite 2-variation along a sequence of partitions $(\pi_n)$,
then $f(\cdot,X_\cdot)$ has the strong (resp. weak) $\alpha$-tensor quadratic variation given by
\begin{equation*}
    {^\alpha [f(\cdot,X_{\cdot}),f(\cdot,X_{\cdot})]}_t
    =
    \int_{\lrbrack 0,t \rbrack}
        D_x f(s-,X_{s-})^{\otimes 2} \mathrm{d}{^\alpha [X,X]}^{\mathrm{c}}_s
    + \sum_{0 < s \leq t} (\Delta f(s,X_s))^{\otimes 2}.
\end{equation*}
As a consequence of this and the It\^{o} formula,
we see that the path $Y_t = \int_0^t Df(A_{s-},X_{s-}) \mathrm{d}X_s$
has $\alpha$-tensor quadratic variation represented as 
\begin{equation*}
    {^\alpha \lbrack Y,Y \rbrack}_t
    = \int_{\lrbrack 0,t \rbrack} Df(A_{s-},X_{s-})^{\otimes 2}\mathrm{d}{^\alpha [X,X]}_s
\end{equation*}
(see Theorem~\ref{8b}),
where $A$ is a {\cadlag} path of finite variation and $f$ is a function with a
certain $C^{1,2}$-smoothness.
We can directly deduce from this theorem that 
the c\`{a}dl\`{a}g path $f(A,X)$ has the
decomposition $f(X) = Y + C + D$, where 
$Y$ is the path defined above, $C$ is a continuous path of 
finite variation, and $D$ is a purely discontinuous path 
of finite variation (Corollary~\ref{8d}).  

Some auxiliary results are provided in the appendices.
Appendix~A gives remarks on vector integration and the RNP.
In Appendix B, we discuss the problem of controlling
the oscillation of a family of {\cadlag} paths by a sequence of partitions.
This discussion is related to Condition~(UC) in Section~5.

\section{Settings}

The aim of this section is to present 
basic concepts in the It\^{o}--F\"{o}llmer calculus in Banach spaces
introduced in the preceding paper~\cite{Hirai_2022a}.

First, we introduce the basic notation and terminology used throughout this paper.
Let $\mathbb{N} = \{ 0,1,2, \dots \}$ and let $\mathbb{R}$ be the set of real numbers.
In this article, the scalar field of a linear space
is always assumed to be $\mathbb{R}$.
Given two Banach spaces $E$ and $F$,
let $\mathcal{L}(E,F)$ denote the space of bounded linear maps from $E$ to $F$.
If, in addition, $G$ is another Banach space,
we define $\mathcal{L}^{(2)}(E,F;G)$ as the 
space of bounded bilinear maps from $E \times F$ to $G$.
As usual, we regard $\mathcal{L}(E,F)$ and $\mathcal{L}^{(2)}(E,F;G)$ as Banach spaces
endowed with the following respective norms:
\begin{equation*}
    \lVert A \rVert = \sup_{x \in E \setminus \{ 0 \}} \frac{\lVert Ax \rVert}{\lVert x \rVert}, \qquad
    \lVert B \rVert = \sup_{x \in E \setminus \{ 0 \}, y \in F \setminus \{ 0 \}} \frac{\lVert B(x,y) \rVert}{\lVert x \rVert \lVert y \rVert}.
\end{equation*}

Let $[0,\infty \rlbrack = \mathbb{R}_{\geq 0} = \{ r \in \mathbb{R} \mid r \geq 0 \}$
and let $E$ be a Banach space.
A \emph{c\`{a}dl\`{a}g} path in $E$
is a function $X \colon \mathbb{R}_{\geq 0} \to E$
that is right-continuous at every $t \geq 0$
and has a left limit at every $t > 0$.
We also use the term \emph{right-regular} to describe the same property.
The symbol $D(\mathbb{R}_{\geq 0},E)$ or $D([0,\infty \rlbrack, E)$ denotes the set of 
all c\`{a}dl\`{a}g paths in $E$.
If $X$ is an $E$-valued {\cadlag} path, we define
\begin{equation*}
    X(t-)
    = \lim_{s \uparrow\uparrow t} X(s)
    \coloneqq \lim_{s \to t, s < t} X(s), \qquad
    \Delta X(t) = X(t) - X(t-).
\end{equation*}
We also use symbols $X_t$, $X_{t-}$, and $\Delta X_t$
to indicate the values $X(t)$, $X(t-)$, and $\Delta X(t)$, respectively.
Next, let
\begin{gather*}
    D(X) = \{ t \in \mathbb{R}_{\geq 0} \mid \lVert \Delta X_t \rVert \neq 0 \}  \\
    D_{\varepsilon}(X) = \{ t \in \mathbb{R}_{\geq 0} \mid \lVert \Delta X_t \rVert > \varepsilon \}  \\
    D^{\varepsilon}(X) = D(X) \setminus D_{\varepsilon}(X)
        = \{ t \in \mathbb{R}_{\geq 0} \mid 0 < \lVert \Delta X_t \rVert \leq \varepsilon \}.
\end{gather*}
Given a set $D \subset \lbrack 0,\infty \rlbrack$
without accumulation points
and a c\`{a}dl\`{a}g path $X$, we define
\begin{equation*}
    J_D(X)_t = J(D;X)_t = \sum_{0 < s \leq t} \Delta X_s 1_{D}(s).
\end{equation*}
Then $J_D(X)$ is a c\`{a}dl\`{a}g path of finite variation.
For abbreviation, we often write
$J_\varepsilon(X)$ instead of $J(D_{\varepsilon}(X);X)$.

Throughout this paper,
the term \emph{partition of $\mathbb{R}_{\geq 0}$} always 
means a set of half-open intervals of the form
$\pi = \{ \lrbrack t_i,t_{i+1} \rbrack; i \in \mathbb{N} \}$
that satisfies $0=t_0 < t_1 < \dots \to \infty$. 
The set of all partitions of $\mathbb{R}_{\geq 0}$ is 
denoted by $\Par(\mathbb{R}_{\geq 0})$
or $\Par (\lbrack 0,\infty \rlbrack )$.
Similarly, $\Par([a,b])$ indicates the set of all partitions of the 
form $\pi = \{ \lrbrack t_i,t_{i+1} \rbrack ; 0 \leq i \leq n-1 \}$
with $a=t_0 < t_1 < \dots < t_n = b$.
If $\pi = \{ \lrbrack t_i,t_{i+1} \rbrack \mid i \in I \}$
is a partition of $\mathbb{R}_{\geq 0}$ or a compact interval,
we set $\pi^{\mathrm{p}} = \{ t_i \mid i \in I \}$.
In other words, $\pi^{\mathrm{p}}$ is the set of 
all endpoints of intervals that belongs to $\pi$.
As usual, define the mesh of a partition $\pi$
by 
$
    \lvert \pi \rvert
    = \sup_{\lrbrack r,s \rbrack \in \pi} \lvert r-s \rvert
$.

Let $\mathcal{I}$ be the semiring of subsets of
$\mathbb{R}_{\geq 0}$ consisting of all intervals of the form
$\lrbrack a,b \rbrack$ ($a \leq b$) and $\{ 0 \}$.
The difference of a path $X \colon \mathbb{R}_{\geq 0} \to E$,
denoted by $\delta X$,
is a function from $\mathcal{I}$ into $E$ defined by
$\delta X(\lrbrack r,s \rbrack) = X(s) - X(r)$ for $0 \leq r \leq s$
and by $\delta X(\{ 0 \}) = X_0$. 
In particular, we have $\delta X(\emptyset) = 0$.
For each $t \in \mathbb{R}_{\geq 0}$,
we also define $\delta X_t \colon \mathcal{I} \to E$ by
$\delta X_t(I)= \delta X (I \cap [0,t])$.
If $I = \lrbrack r,s \rbrack \in \mathcal{I}$,
then $\delta X_t(I) = X(s \wedge t) - X(r \wedge t)$.
Next, consider a bilinear map $B \colon F \times E \to G$
between Banach spaces and another path
$Y \colon \mathbb{R}_{\geq 0} \to F$.
Then we define functions $B(Y,\delta X_t)$, $B(\delta Y_t,X)$,
and $B(\delta Y_t,\delta X_t)_t$ from $\mathcal{I}$ to $G$
by the formulae
\begin{gather*}
    B(Y,\delta X_t)(\lrbrack r,s \rbrack) = B(Y_r,\delta X_t(\lrbrack r,s \rbrack)),    \\
    B(\delta Y_t,X)(\lrbrack r,s \rbrack) = B(\delta Y_t(\lrbrack r,s \rbrack),X_r),    \\
    B(\delta Y_t,\delta X_t)(\lrbrack r,s \rbrack) = B(\delta Y_t(\lrbrack r,s \rbrack),\delta X_t (\lrbrack r,s \rbrack))
\end{gather*}
for $I = \lrbrack r,s \rbrack \in \mathcal{I}$ and
\begin{equation*}
    B(Y,\delta X_t)(\{ 0 \})
    = B(\delta Y_t,X)(\{ 0 \})
    = B(\delta Y_t,\delta X_t)(\{ 0 \})
    = B(Y_0,X_0).
\end{equation*}
By this notation, the left-side Riemannian sum
has a relatively shorter expression
\begin{equation*}
    \sum_{\lrbrack r,s \rbrack \in \pi} B(Y_r, X_{t \wedge s} - X_{t \wedge r})
    =
    \sum_{I \in \pi} B(Y,\delta X_t)(I).
\end{equation*}

Let $X$ be a {\cadlag} path in a Banach space $E$.
We say that $X$ has finite variation if
\begin{equation*}
    V(X;[a,b])
    =
    \sup_{\pi \in \Par([a,b])} \sum_{I \in \pi} 
        \lVert \delta X(I) \rVert_E < \infty
\end{equation*}
for all compact intervals $[a,b] \subset \mathbb{R}_{\geq 0}$.
The set of all $E$-valued {\cadlag} paths
is denoted by $FV(\mathbb{R}_{\geq 0},E)$.
For each $t \geq 0$, we define $V(X)_t = V(X;[0,t])$
and call $t \mapsto V(X)_t$ the total variation path of $X$.
If $X$ has finite variation, then $V(X)$
is an increasing {\cadlag} path with $V(X)_0 = 0$.

\begin{Def} \label{2b}
Let $E,F,G$ be Banach spaces
and let $B \in \mathcal{L}^{(2)}(E,F;G)$.
For a path $(X,Y) \in D(\mathbb{R}_{\geq 0},E \times F)$
and a partition $\pi$ of $\mathbb{R}_{\geq 0}$,
define the discrete quadratic covariation as
\begin{equation*}
    Q_B^{\pi}(X,Y)_t
    = \sum_{I \in \pi} B(\delta X_t,\delta Y_t)(I), \qquad t \geq 0.
\end{equation*}
Now let $(\pi_n)$ be a sequence of
partitions of $\mathbb{R}_{\geq 0}$.
We say that $(X,Y)$ has
\emph{strong $B$-quadratic covariation along $(\pi_n)_{n \in \mathbb{N}}$}
if there exists a $G$-valued c\`{a}dl\`{a}g path
$Q_B(X,Y)$ of finite variation
that satisfies the following conditions:
\begin{enumerate}
    \item the sequence $(Q^{\pi_n}_B(X,Y))_{n \in \mathbb{N}}$ converges to $Q(X,Y)$ pointwise in the norm topology;
    \item the jump of $Q_B(X,Y)$ satisfies $\Delta Q_B(X,Y)_t = B(\Delta X_t,\Delta Y_t)$
        for all $t \geq 0$.
\end{enumerate}
Then $Q_B(X,Y)$ is called
the strong $B$-quadratic covariation of $(X,Y)$.
If the convergence in (i) is replaced with weak convergence,
we say that $X$ has the weak $B$-quadratic covariation $Q_B(X,Y)$.
\end{Def}

If $E=F$ and $X=Y$,
we simply call $Q_B(X,X)$ the strong or weak
$B$-quadratic variation of $X$.

Let $\alpha$ be a reasonable crossnorm on $E \otimes F$
and $E \widehat{\otimes}_{\alpha} F$ be the 
completion of normed space $(E \otimes F,\alpha)$.
See Diestel and Uhl~\cite[221--222]{Diestel_Uhl_1977} 
or Ryan~\cite[127]{Ryan_2002} for the definition.
For the strong (resp. weak) quadratic covariation
with respect to the canonical bilinear map
$\otimes \colon E \times F \to E \widehat{\otimes}_{\alpha} F$,
we use the symbol ${^\alpha [X,Y]}$
and call it the strong (resp. weak)
\emph{$\alpha$-tensor quadratic covariation}.
If $E=F$ and $X=Y$, the path ${^\alpha [X,X]}$ is called
the (strong/weak) $\alpha$-tensor quadratic variation.
Recall that there is the greatest crossnorm $\gamma$,
which is usually called the projective norm,
on the tensor product of any two Banach spaces.
For this special crossnorm, we often omit the symbol $\gamma$
and write $E \widehat{\otimes} E = E \widehat{\otimes}_{\gamma} E$
and $[X,Y] = {^\gamma [X,Y]}$.
Then we call $[X,Y] \colon \mathbb{R}_{\geq 0} \to E \widehat{\otimes} E$
the projective tensor quadratic covariation.

Now we introduce a different type of quadratic variation,
namely, scalar quadratic variation.
First, define the discrete scalar
quadratic variation of
$X \colon \mathbb{R}_{\geq 0} \to E$
along a partition $\pi$ by 
\begin{equation*}
    Q^{\pi}(X)_t = \sum_{I \in \pi} \lVert \delta X_t(I) \rVert^2.
\end{equation*}

\begin{Def} \label{2c}
Let $X$ be a c\`{a}dl\`{a}g path in a Banach space $E$
and $(\pi_n)$ be a sequence of partitions of $\mathbb{R}_{\geq 0}$.
\begin{enumerate}
    \item The path $X$ has \emph{finite 2-variation along $(\pi_n)$} if
        \begin{equation*}
            V^{(2)}(X;(\pi_n))_t
            \coloneqq
            \sup_{n \in \mathbb{N}} Q^{\pi_n}(X)_t
            < \infty
        \end{equation*}
        for all $t \in \mathbb{R}_{\geq 0}$.
    \item The path $X \colon \mathbb{R}_{\geq 0} \to E$ has
        \emph{scalar quadratic variation}
        if there exists a real-valued c\`{a}dl\`{a}g increasing path
        $Q(X)$ satisfying the following conditions:
        \begin{enumerate}
            \item the sequence $Q^{\pi_n}(X)$ converges to $Q(X)$ pointwise;
            \item the equality $\Delta Q(X)_t = \lVert \Delta X_t \rVert^2$ 
                holds for all $t \geq 0$.
        \end{enumerate}
        Then we call $Q(X)$ the scalar quadratic variation
        of $X$ along $(\pi_n)$.
\end{enumerate}
\end{Def}

Condition~(i) is much weaker than asserting that $X$ has finite 2-variation
in the usual sense, that is,
\begin{equation*}
    \sup_{\pi \in \Par([a,b])} \sum_{I \in \pi} 
        \lVert \delta X(I) \rVert^2 < \infty
\end{equation*}
for all compact intervals $[a,b] \subset \mathbb{R}_{\geq 0}$.
If $X$ has scalar quadratic variation along $(\pi_n)$,
then it has finite 2-variation along the same sequence $(\pi_n)$.
Note that if $E$ is a Hilbert space,
the scalar quadratic variation coincides with the quadratic variation
$Q_{\langle\phantom{x},\phantom{x}\rangle}(X,X)$,
where $\langle\phantom{x},\phantom{x}\rangle \colon E \times E \to \mathbb{R}$
is the inner product of $E$.

A typical example of tensor and scalar quadratic variations is
the tensor and scalar quadratic variations of a semimartingale in 
a Hilbert space.
See, for example, Metivier and Pellaumail~\cite{Metivier_Pellaumail_1980b}
and Metivier~\cite{Metivier_1982}.
A Banach space-valued path with finite variation has
tensor and scalar quadratic variations along any sequence of partitions
satisfying $\lvert \pi_n \rvert \to 0$.
It is also true for a partition satisfying a slightly general condition
in Definition~\ref{2d} below.
See Hirai~\cite[Section 5]{Hirai_2022a} for a proof.
Furthermore, for $\alpha > 1/2$,
an $\alpha$-H\"{o}lder continuous path in a Banach space 
has tensor and scalar quadratic variations, which are identically zero,
along a sequence with the condition $\lvert \pi_n \rvert \to 0$.

Next, we introduce conditions on a sequence of partitions
that is required for some important theorems, including the It\^{o} formula.  
Let $\pi \in \Par(\mathbb{R}_{\geq 0})$ and $t \in \lrbrack 0,\infty \mathclose{[}$.
The symbol $\pi(t)$ denotes the element of $\pi$ that contains $t$.
By definition, there exists only one such interval.
Moreover, we define $\overline{\pi}(t) = \sup \pi(t)$ and $\underline{\pi}(t) = \inf \pi(t)$.
Then we have $\pi(t) = \lrbrack \underline{\pi}(t),\overline{\pi}(t) \rbrack$, and 
\begin{gather*}
    \delta X_t(\pi(s)) = X(\overline{\pi}(s) \wedge t) - X(\underline{\pi}(s) \wedge t), \qquad
    \delta X(\pi(s)) = X(\overline{\pi}(s)) - X(\underline{\pi}(s))
\end{gather*}
hold for all $s$ and $t$ in $\lrbrack 0,\infty \rlbrack$.

Let $f \colon S \to E$ be a function
into a Banach space and set
\begin{equation*}
    \omega(f,A) = \sup_{x,y \in A} \lVert f(x) - f(y) \rVert_E
\end{equation*}
for each subset $A$ of $S$.
Using this notation, we define oscillations of
$X \in D(\mathbb{R}_{\geq 0},E)$  along $\pi \in \Par(\mathbb{R}_{\geq 0})$ by
\begin{gather*}
    O^+_{t}(X,\pi)
    = \sup_{\lrbrack r,s \rbrack \in \pi}
        \omega(X, \lrbrack r,s \rbrack \cap [0,t]),   \\
    O^-_t(X,\pi)
    = \sup_{\lrbrack r,s \rbrack \in \pi}
        \omega(X, \lbrack r,s \rlbrack \cap [0,t]).
\end{gather*}
By the right continuity,
we have $O^-_t(X,\pi) \leq O^+_t(X,\pi)$.
Two oscillations $O^-$ and $O^+$ coincide if $X$ is continuous.

In the It\^{o}--F\"{o}llmer calculus,
one often requires the sequence $(\pi_n)$
to satisfy either $\lim_{n \to \infty} \lvert \pi_n \rvert = 0$
or $O^-_t(X,\pi_n) \to 0$ for all $t \geq 0$.
When considering {\cadlag} paths, neither 
includes the other.
In this paper,
we work with another assumption
given in Definition~\ref{2d},
which unifies these two approaches.

\begin{Def}[Hirai~\cite{Hirai_2022a}] \label{2d}
Let $X$ be a {\cadlag} path in a Banach space $E$
and let $(\pi_n)_{n \in \mathbb{N}}$
be a sequence of partitions of $\mathbb{R}_{\geq 0}$.
\begin{enumerate}
\item The sequence $(\pi_n)$ satisfies Condition~(C) for $X$
    if it satisfies the following three conditions:
	\begin{enumerate}
        \item[(C1)] Let $t \in \mathbb{R}_{\geq 0}$ and $\varepsilon > 0$. 
            Then there exists an $N \in \mathbb{N}$ such that
            $I \cap [0,t] \cap D_{\varepsilon}(X)$
            has at most one element
            for all $n \geq N$ and all $I \in \pi_n$.
        \item[(C2)] The sequence $(\delta X_t(\pi_n(s)))_{n \in \mathbb{N}}$
            converges to $\Delta X_s$ for all $s \in D(X)$ and $t \in [s,\infty \rlbrack$.
        \item[(C3)] For all $t \in \mathbb{R}_{\geq 0}$,
            \begin{equation*}
            \varlimsup_{\varepsilon \downarrow\downarrow 0} \varlimsup_{n \to \infty} O^{+}_t(X-J_{\varepsilon}(X);\pi_n) = 0.
            \end{equation*}
    \end{enumerate}
    We say that $(\pi_n)$ satisfies (C) for $X$ on $[0,T]$
    if $(\pi_n)$ satisfies (C) for the stopped path $X(\cdot \wedge T)$.
\item The sequence \emph{$(\pi_n)$ approximates $X \colon \mathbb{R}_{\geq 0} \to E$ from the left}
    if $\lim_{n \to \infty} X(\underline{\pi_n}(t)) = X(t-)$ holds for all $t > 0$.
    Then we call $(\pi_n)$ a \emph{left approximation sequence}
    for $X$.
\end{enumerate}
\end{Def}

In the framework of the It\^{o}--F\"{o}llmer calculus in Banach spaces described above,
we have a $C^{1,2}$-type It\^o formula.
We first quote a generalized form of the theorem (Theorem~\ref{2e}),
which treats general bilinear quadratic variations.
In what follows,
let $E$, $E_1$, $F$, and $G$ be Banach spaces
and $B \colon E \times E \to E_1$ be a bounded bilinear map.
The symbol $C_{\mathcal{K}}(F \times E,G')$
denotes the set of all functions $f \colon F \times E \to G'$
into a Banach space $G'$ such that $f\vert_{K}$ is continuous
whenever $K \subset F \times E$ is compact.

As we can see in Corollary~\ref{2fc},
a typical example of $B$ is the canonical bilinear map
$\otimes \colon E \times E \to E \widehat{\otimes}_{\alpha} E$
for some reasonable crossnorm $\alpha$.
A more elementary case is given by $f = F \circ P$
with $P \colon E \to V$ being a finite-dimensional projection
and $F$ is a $C^2$ function in the usual sense.

\begin{Th}[Hirai~\cite{Hirai_2022a}] \label{2e}
Let $(A,X) \in D(\mathbb{R}_{\geq 0},F \times E)$
and $(\pi_n)$ be a sequence of partitions
that satisfies Condition~(C) for $(A,X)$ and approximate it from the left.
Suppose that $X \in D(\mathbb{R}_{\geq 0},E)$
has weak $B$-quadratic variation and finite 2-variation along $(\pi_n)$
and that $A \in D(\mathbb{R}_{\geq 0},F)$ has finite variation.
Consider a function $f \colon F \times E \to G$ that is
twice G\^{a}teaux differentiable in the second variable $x$
and once in the first variable $a$.
In addition, suppose that $D_a f \in C_{\mathcal{K}}(F \times E,\mathcal{L}(F,G))$,
$D_x f \in C_{\mathcal{K}}(F \times E,\mathcal{L}(E,G))$,
and there exists a $D_B^2 f \in C_{\mathcal{K}}(F \times E,\mathcal{L}(E_1,G))$
such that $D^2_x f(a,x) = D^2_Bf(a,x) \circ B$ for all $(a,x) \in F \times E$.
Then the path $t \mapsto f(A_t,X_t)$ satisfies
\begin{multline} \label{2f}
    f(A_t,X_t) - f(A_0,X_0)
    = 
    \int_{0}^{t} D_a f (A_{s-},X_{s-}) \mathrm{d}A^{\mathrm{c}}_s
    + \int_0^t D_x f(A_{s-},X_{s-})\mathrm{d}X_s   \\
    + \frac{1}{2} \int_{0}^{t} D_B^2 f(A_{s-},X_{s-})\mathrm{d}Q_B(X,X)^{\mathrm{c}}_s
    + \sum_{0 < s \leq t} \left\{ \Delta f(A_s,X_s) - D_x f(A_{s-},X_{s-})\Delta X_s \right\}
\end{multline}
for all $t \geq 0$,
where the first and the third integrals on the right-hand side
are usual vector Stieltjes integrals on $\lrbrack 0,t \rbrack$
and the second one is the It\^{o}--F\"{o}llmer integral
defined as the weak limit
\begin{equation} \label{2fb}
    \int_0^t D_x f(A_{s-},X_{s-})\mathrm{d}X_s
    =
    \lim_{n \to \infty} \sum_{I \in \pi_n}
        D_x f(A,X) \delta X_t(I).
\end{equation}
Moreover, if the quadratic variation $Q_B$ exits in the strong sense,
the convergence of \eqref{2fb} holds in the norm topology of $G$.
\end{Th}

See the first paper in this series~\cite{Hirai_2022a} for a proof.
Note that in \cite{Hirai_2022a}, we assumed the continuity of
derivatives of $f$ on the whole space $F \times E$.
We can indeed weaken the continuity of derivatives as in Theorem~\ref{2e}
without modifying the proof.

Applying Theorem~\ref{2e} to $E_1 = E \widehat{\otimes}_{\alpha} E$ and 
$B = \otimes \colon E \times E \to E \widehat{\otimes}_{\alpha} E$
for a reasonable crossnorm $\alpha$,
we obtain the following corollary.

\begin{Cor} \label{2fc}
Let $X \in D(\mathbb{R}_{\geq 0},E)$ be a path with
weak $\alpha$-tensor quadratic variation
and finite 2-variation along $(\pi_n)$,
and let $A \in FV(\mathbb{R}_{\geq 0},F)$.
Assume that a sequence of partitions $(\pi_n)$ satisfies (C)
and is a left-approximation sequence for $(A,X)$.
Let $f \colon F \times E \to G$ be a function
that is once G\^{a}teaux differentiable in $a$ and twice
G\^{a}teaux differentiable in $x$.
Suppose moreover that 
$D_a f \in C_{\mathcal{K}}(F \times E,\mathcal{L}(F,G))$,
$D_x f \in C_{\mathcal{K}}(F \times E,\mathcal{L}(E,G))$,
and $D_x^2 f \colon F \times E \to \mathcal{L}(E \widehat{\otimes} E,G)$
extends to a function in
$C_{\mathcal{K}}(F \times E,\mathcal{L}(E \widehat{\otimes}_{\alpha} E,G))$.
Then $f(A,X)$ satisfies the following It\^o formula:
\begin{multline*}
    f(A_t,X_t) - f(A_0,X_0)
    = 
    \int_{0}^{t} D_a f (A_{s-},X_{s-}) \mathrm{d}A^{\mathrm{c}}_s
    + \int_0^t D_x f(A_{s-},X_{s-})\mathrm{d}X_s   \\
    + \frac{1}{2} \int_{0}^{t} D_x^2 f(A_{s-},X_{s-})\mathrm{d}{^\alpha [X,X]}^{\mathrm{c}}_s
    + \sum_{0 < s \leq t} \left\{ \Delta f(A_s,X_s) - D_x f(A_{s-},X_{s-})\Delta X_s \right\},
\end{multline*}
where the second integral on the right-hand side is
the It\^{o}--F\"{o}llmer defined as the weak limit.
Moreover, if ${^\alpha [X,X]}$ is the strong $\alpha$-tensor quadratic variation,
then the It\^{o}--F\"{o}llmer integral converges in the norm topology.
\end{Cor}

If $\alpha$ is the projective norm,
then the assumption on the second derivative
can be simply written as
$D^2_x f \in C_{\mathcal{K}}(F \times E,\mathcal{L}(E \widehat{\otimes} E,G))$.
This is because $D_x^2 f(a,x)$ is generally an element of $\mathcal{L}(E \widehat{\otimes} E,G)$
but not of $\mathcal{L}(E \widehat{\otimes}_{\alpha} E,G)$.

The following lemma, which is essential for the proof of Theorem~\ref{2e},
will also be used in Section~5 to prove one of the main theorems of this paper.

\begin{Lem} \label{2g}
Let $(\pi_n)$ be a sequence in partitions and
let $X \in D(\mathbb{R}_{\geq 0},E)$ be a path with
strong or weak $B$-quadratic variation and finite 2-variation along $(\pi_n)$.
Assume that $(\pi_n)$ satisfies Condition~(C) for $X$ and 
approximates a $\xi \in D(\mathbb{R}_{\geq 0},\mathcal{L}(E_1,G))$ from the left.
Then for all $t \in \mathbb{R}_{\geq 0}$,
\begin{equation*}
    \lim_{n \to \infty}
        \sum_{I \in \pi_n} \xi \delta X_t(I)
    =
    \int_{\lrbrack 0,t \rbrack} \xi_{s-}\mathrm{d}Q_B(X,X)_s,
\end{equation*}
holds in the respective topology.
\end{Lem}

\section{Linear transformations}

In this section, we give some simple results 
related to linear transformations of paths and quadratic variations.

\begin{Prop} \label{3b}
Let $E$, $F$, and $G$ be Banach spaces and let $B \in \mathcal{L}^{(2)}(E,F;G)$.
Suppose that $(X,Y) \in D(\mathbb{R}_{\geq 0},E \times F)$ has
strong (resp. weak) $B$-quadratic covariation
along a sequence of partitions $(\pi_n)$.
Then for every Banach space $G'$ and $T \in \mathcal{L}(G,G')$,
the pair $(X,Y)$ has
strong (resp. weak) $T \circ B$-quadratic covariation along $(\pi_n)$,
given by
$Q_{T \circ B}(X,Y) = T \circ Q_B(X,Y)$.
\end{Prop}

\begin{proof}
First, note that
$V(T \circ Q_B(X,Y);I) \leq \lVert T \rVert V(Q_B(X,Y);I)$
holds for all of compact interval $I$,
and therefore $T \circ Q_B(X,Y)$ is a {\cadlag}
path of finite variation.
By the linearity and the continuity of $T$, we see that
\begin{equation*}
    \sum_{I \in \pi_n} T (B(\delta X_t,\delta Y_t))(I)
    =
    T \left( \sum_{I \in \pi_n} B(\delta X_t,\delta Y_t)(I) \right)
    \xrightarrow[n \to \infty]{} T(Q_B(X,Y)_t)
\end{equation*}
holds for all $t \geq 0$ in the suitable topology.
Moreover, we have
\begin{equation*}
    \Delta T(Q_B(X,Y)_t)
    =
    T(\Delta Q_B(X,Y)_t) 
    =
    T(B(\Delta X_t,\Delta Y_t)). 
\end{equation*}
Hence, $T \circ Q_B(X,Y)$ is the $T \circ B$-
quadratic covariation of $(X,Y)$.
\end{proof}

As a consequence of Proposition~\ref{3b},
we can reveal the relation between projective tensor
quadratic variation and other quadratic variations
with respect to bilinear maps.
Now recall that there is an isometric isomorphism
$\mathcal{L}^{(2)}(E,F;G) \cong \mathcal{L}(E \widehat{\otimes} F,G)$.

\begin{Prop} \label{3c}
Let $(\pi_n)$ be a sequence of partitions 
and let $(X,Y) \in D(\mathbb{R}_{\geq 0},E \times F)$.
Then the following conditions are equivalent:
\begin{enumerate}
    \item $(X,Y)$ has strong (resp. weak) projective tensor quadratic covariation.
    \item $(X,Y)$ has strong (resp. weak) $B$-quadratic covariation
        for every $B \in \mathcal{L}^{(2)}(E,F;G)$.
\end{enumerate}
If these conditions are satisfied, then 
$T_B \circ [X,Y] = Q_B(X,Y)$ holds for every $B \in \mathcal{L}(E,F;G)$,
where $T_B$ is the unique bounded linear map
that commutes the following diagram.
\begin{equation*}
    \begin{tikzcd}
        E \times F \arrow[r,"B"] \arrow[d,"\otimes"'] & G   \\
        E \widehat{\otimes} F \arrow[ru,"T_B"']
    \end{tikzcd}
\end{equation*}
\end{Prop}

\begin{proof}
First, assume that $X$ and $Y$ have tensor quadratic covariation
along $(\pi_n)$.
Then, by Proposition~\ref{3b},
$X$ and $Y$ have quadratic covariation with respect to 
$B = T_B \circ \otimes$, given by $T_B \circ [X,Y] = Q_B(X,Y)$.
Conversely, if condition~(ii) holds,
we get condition~(i) by applying (ii) to the bounded bilinear map
$\otimes \colon E \times F \to E \widehat{\otimes} F$.
\end{proof}

\begin{Cor} \label{3e}
Suppose that $(X,Y)$ has the weak projective
tensor quadratic covariation along $(\pi_n)$.
Then for any Banach space $G$,
the map $B \mapsto Q_B(X,Y)$ from $\mathcal{L}^{(2)}(E,F;G)$
to $FV(\mathbb{R}_{\geq 0}, G)$ 
is continuous with respect to the topology of pointwise convergence in
$FV(\mathbb{R}_{\geq 0}, G)$.
\end{Cor}

\begin{proof}
Recall that the canonical isomorphism
$\mathcal{L}(E,F; G) \rightarrow \mathcal{L}(E \widehat{\otimes} F, G)$
is isometric with respect to the operator norm.
Combining this fact with Proposition~\ref{3c}, we see that
\begin{equation*}
    \lVert Q_B(X,Y)_t - Q_{B'}(X,Y)_t \rVert
\leq 
    \lVert B - B' \rVert_{\mathcal{L}^{(2)}(E,F;G)} \lVert [X,Y]_t \rVert_{E \widehat{\otimes} F}.
\end{equation*}
This shows the desired continuity.
\end{proof}

\begin{Prop} \label{3f}
Let $E_i$, $F_i$ ($i = 1,2$), and $G$ be Banach spaces.
We assume that $T_i \in \mathcal{L}(E_i,F_i)$ ($i \in \{1,2 \}$), 
$B \in \mathcal{L}^{(2)}(E_1,E_2;G)$,
and $B' \in \mathcal{L}^{(2)}(F_1,F_2;G)$ satisfy
$B' \circ (T_1 \times T_2) = B$.
\begin{equation*}
    \begin{tikzcd}[column sep=small]
        E_1 \times E_2 \arrow[rd,"B"'] \arrow[rr,"T_1 \times T_2"] & & F_1 \times F_2 \arrow[ld,"{B'}"]  \\
        & G &
    \end{tikzcd}
\end{equation*}
If $(X_1,X_2) \in D(\mathbb{R}_{\geq 0},E_1 \times E_2)$
has strong (resp. weak) $B$-quadratic covariation, 
$(T_1 \circ X_1,T_2 \circ X_2)$ has 
the strong (resp. weak) $B'$-quadratic covariation, given by
$Q_{B'}(T_1 \circ X_1,T_2 \circ X_2) = Q_B (X_1,X_2)$.
\end{Prop}

\begin{proof}
By a direct calculation, we get
\begin{equation*}
    B'(\delta (T_1 \circ X_1)_t, \delta (T_2 \circ X_2)_t)(I)
    =
    B(\delta (X_1)_t, \delta (X_2)_t)(I).
\end{equation*}
Therefore
\begin{equation*}
    \lim_{n \to \infty} 
    \sum_{\lrbrack r,s \rbrack \in \pi_n} 
        B'(\delta (T_1 \circ X_1)_t, \delta (T_2 \circ X_2)_t)(I)
    =
    Q_B(X_1,X_2)_t
\end{equation*}
holds in the suitable topology.
Moreover, 
\begin{equation*}
    \Delta Q_B(X_1,X_2)_t
    =
    B(\Delta (X_1)_t,\Delta (X_2)_t) 
    =
    B'\left( \Delta T_1(X_1)_t,\Delta T_2 (X_2)_t \right)
\end{equation*}
for all $t \in \mathbb{R}_{\geq 0}$.
Hence, $Q_B$ is the $B'$-quadratic covariation of
$T_1 \circ X_1$ and $X_2 \circ X_2$.
\end{proof}

It follows from above results that 
every $(X,Y)$ with tensor quadratic variation
has `cylindrical' quadratic covariation.

\begin{Cor} \label{3fb}
Let $E$ and $F$ be Banach spaces and let $\alpha$ be a reasonable crossnorm on $E \otimes F$.
Suppose that $(X,Y) \in D(\mathbb{R}_{\geq 0},E \times F)$ has
weak $\alpha$-tensor quadratic variation along a sequence of partitions $(\pi_n)$.
Then, for each $(x^\ast,y^\ast) \in E^\ast \times F^\ast$,
real-valued paths $x^\ast X$ and $y^\ast Y$ have
the quadratic covariation
$[x^\ast X, y^\ast Y] = Q_{x^\ast \otimes y^\ast}(X,Y) = x^\ast \otimes y^\ast([X,Y])$
along $(\pi_n)$.
Here, $x^\ast \otimes y^\ast \colon X \widehat{\otimes}_{\alpha} Y \to \mathbb{R}$
is the bounded linear form defined to be
$(x^\ast \otimes y^\ast)(x \otimes y) = x^\ast(x) y^\ast(y)$
for all $(x,y) \in X \times Y$.
\end{Cor}

\begin{proof}
The equations $[x^\ast X, y^\ast Y] = Q_{x^\ast \otimes y^\ast}(X,Y)$
and $Q_{x^\ast \otimes y^\ast}(X,Y) = x^\ast \otimes y^\ast([X,Y])$
follow from Propositions~\ref{3f} and \ref{3b}, respectively.
\end{proof}

To derive the trace representation formula for the scalar quadratic variations,
let us recall the definition of the trace operator.
Let $(H,\langle \phantom{x},\phantom{x} \rangle_H)$ be a Hilbert space.
The trace operator $\mathop{\mathrm{tr}}\nolimits_H \colon H \otimes H \to \mathbb{R}$
is a unique linear form satisfying 
$\mathop{\mathrm{tr}}\nolimits_H(x \otimes y) = \langle x,y \rangle_H$ for all $x,y \in H$.
Since $\operatorname{tr}_H$ is contractive with respect to the 
projective norm on $H \otimes H$, 
it can be uniquely extended to the completion $H \widehat{\otimes} H$.

\begin{Prop} \label{3g}
Let $H$ be a Hilbert space
and $X \colon \mathbb{R}_{\geq 0} \to H$ a c\`{a}dl\`{a}g path.
If $X$ has weak projective tensor quadratic variation
along a sequence of partitions $(\pi_n)$,
then it has the scalar quadratic variation given by 
$Q(X)_t = \operatorname{tr}_H([X,X]_t)$ for all $t$.
\end{Prop}

\begin{proof}
Applying Proposition~\ref{3c}
to the bounded bilinear map
$\langle \phantom{x},\phantom{x} \rangle_H = \operatorname{tr}_H \circ \otimes$,
we see that $Q(X)_t = Q_{\langle \phantom{x},\phantom{x} \rangle_H}(X,X)_t = \operatorname{tr}_H([X,X]_t)$
holds for all $t \geq 0$.
\end{proof}

\section{Integral representation of quadratic variations by means of scalar quadratic variation}

In this section, we show integral representation formulae
of $B$-quadratic variations with respect
to the scalar quadratic variation.
For a similar result in the
classical martingale theory in Hilbert spaces,
see Metivier~\cite[Section 21]{Metivier_1982}.
Throughout this section,
let $E$ and $G$ be Banach spaces
and $B \colon E \times E \to G$ be a bounded bilinear map.
Moreover, let $(\pi_n)$ be a sequence of partitions of $\mathbb{R}_{\geq 0}$.

\begin{Prop} \label{4d}

Suppose that $X \in D(\mathbb{R}_{\geq 0},E)$
has weak $B$-quadratic variation and 
finite 2-variation along $(\pi_n)$
and that $(\pi_n)$ satisfies Condition~(C) for $X$.
Then
\begin{equation} \label{4e}
    \left\lVert Q_B(X,X)_t - Q_B(X,X)_s \right\rVert_G \leq \lVert B \rVert (Q(X)_t - Q(X)_s)
\end{equation}
for all $s,t \in \mathbb{R}_{\geq 0}$ with $s \leq t$. Consequently, $Q_B(X,X)$ is absolutely continuous with respect to $Q(X)$.
\end{Prop}

To prove Proposition~\ref{4d}, we present some technical lemmas.

\begin{Lem} \label{4eb}
If $(\pi_n)$ satisfies (C) for $X \in D(\mathbb{R}_{\geq 0},E)$,
then
\begin{gather*}
    \lim_{n \to \infty}
        (X_{\overline{\pi_n}(s) \wedge t} - X_{s})
        \otimes (X_{s} - X_{\underline{\pi_n}(s)})
    = 0
\end{gather*}
holds for all $s,t \in \mathbb{R}_{\geq 0}$ satisfying $0 < s < t$.
\end{Lem}

\begin{proof}
Fix two real numbers $s$ and $t$ such that $0 < s < t$.

\emph{Case 1: $X$ is discontinuous at $s$.}
If $s$ is a discontinuous point of $X$,
Condition~(C2) implies
\begin{equation*}
    X_{\overline{\pi_n}(s) \wedge t} - X_{s}
    =
    (X_{\overline{\pi_n}(s) \wedge t} - X_{\underline{\pi_n}(s) \wedge t})
    - (X_{\overline{\pi_n}(s) \wedge s} - X_{\underline{\pi_n}(s) \wedge s}) 
    \xrightarrow[n \to \infty]{} 
    \Delta X_s - \Delta X_s
    =
    0.
\end{equation*}

\emph{Case 2: $X$ is continuous at $s$.}
First, note that 
$\overline{\pi_n}(s) \to s$ holds
if $D(X) \cap [s,s+\varepsilon]$ has infinitely many elements
for every $\varepsilon > 0$.
This follows from Condition~(C1).
Similarly, $\underline{\pi_n}(s) \to s$ holds
if $D(X) \cap [s-\varepsilon,s ]$ has infinitely many elements
for every $\varepsilon > 0$.
In both cases, we have the desired convergence.

Next, assume $[s-\alpha,s+\alpha]$
contains finitely many points of $D(X)$,
where $0 < \alpha < s$.
Then there are three cases to be considered:
\begin{enumerate}
    \item[A.] $\omega(X;[s,s+\varepsilon]) > 0$ 
        for all $\varepsilon \in \lrbrack 0,\alpha \rlbrack$;
    \item[B.] $\omega(X;[s-\varepsilon,s]) > 0$
        for all $\varepsilon \in \lrbrack 0,\alpha \rlbrack$;
    \item[C.] $\omega(X;[s-\varepsilon,s+\varepsilon]) = 0$ 
        for some $\varepsilon \in \lrbrack 0,\alpha \rlbrack$.
\end{enumerate}

\emph{Case~A.}
Take an arbitrary $\varepsilon \in \lrbrack 0,\alpha \rlbrack$
satisfying $\omega(X;[s,s+\varepsilon]) > \sup_{u \in [s,s+\varepsilon]} \lVert \Delta X_u \rVert$.
The existence of such an $\varepsilon$ follows from the assumption 
that $[s-\alpha,s+\alpha] \cap D(X)$ is finite.
Since there are only finitely many elements in $[s,s+\varepsilon] \cap D(X)$,
we can choose an $N \in \mathbb{N}$ such that 
$\pi_n(s) \cap D(X) \cap [s,s+\varepsilon]$ has at most one element 
for all $n \geq N$.
If $n \geq N$ and $[s,s+\varepsilon] \subset \pi_n(s)$, we have
\begin{align*}
    \omega(X;[s,s+\varepsilon])
    & \leq 
    \omega(X-J_{\delta}(X);[s,s+\varepsilon]) 
    + \sup_{u \in [s,s+\varepsilon]} \lVert \Delta X_u \rVert  \\
    & \leq 
    \omega(X-J_{\delta}(X);\pi_n(s) \cap [0,s+\alpha])
    + \sup_{u \in [s,s+\varepsilon]} \lVert \Delta X_{u} \rVert
\end{align*}
for any positive $\delta$, and hence 
\begin{equation*}
    O_{s+\alpha}^{+}(X-J_{\delta}(X);\pi_n)
    \geq
    \omega(X;[s,s+\varepsilon]) - \sup_{u \in [s,s+\varepsilon]} \lVert \Delta X_{u} \rVert  
    > 
    0
\end{equation*}
holds under the same condition.
Combining this estimate with Condition~(C3),
we see that there are not infinitely many 
$n$ satisfying $[s,s+\varepsilon] \subset \pi_n(s)$.
In other words, we have $\overline{\pi_n}(s) \in [s,s+\varepsilon]$
for sufficiently large $n$.
Since $\varepsilon$ is chosen arbitrarily,
we obtain the convergence $\overline{\pi}_n(s) \to s$ in this case.
Therefore, we obtain the desired convergence.

\emph{Case~B.}
In this case, we can deduce that $\underline{\pi_n}(s) \to s$
by a discussion similar to that for Case~A.

\emph{Case~C.}
Set
\begin{equation*}
    s' = \sup \{ u \leq s \mid X(u) \neq X(s) \}, \qquad 
    s'' = \inf \{ u \geq s \mid X(u) \neq X(s) \}.
\end{equation*}
Then $s' < s < s''$ holds by the assumption.

\emph{Case~C-1: $X$ is continuous at $s'$.}
In this case,
we can deduce by the same argument as that for Case~B that
for any $\varepsilon > 0$,
there is an $N$ satisfying
$\underline{\pi_n}(s) \in \lbrack s'-\varepsilon, s \rlbrack$
for all $n \geq N$.
Therefore,
\begin{equation*}
    \lim_{n \to \infty} \left( X_s - X_{\underline{\pi_n}(s)} \right) = 0.
\end{equation*}

\emph{Case~C-2: $X$ is continuous at $s''$.}
Similarly, we have
\begin{equation*}
    \lim_{n \to \infty} (X_{\overline{\pi_n}(s) \wedge t} - X_{s}) = 0.
\end{equation*}

\emph{Case~C-3: $X$ is discontinuous at both $s'$ and $s''$. }
Let $\delta = \lVert \Delta X_{s'} \rVert \wedge \lVert \Delta X_{s''} \rVert$.
Then choose an $N \in \mathbb{N}$
such that $I \cap [0,t] \cap D_{\delta/2}(X)$
has at most one element for every $I \in \pi_n$
and every $n \geq N$.
If $n \geq N$, we have either
$\overline{\pi_n}(s) \in \lbrack s,s'' \rlbrack$
or $\underline{\pi_n}(s) \in \lbrack s',s \rlbrack$.
In both cases, we get
\begin{equation*}
    (X_{\overline{\pi_n}(s) \wedge t} - X_{s})
    \otimes (X_{s} - X_{\underline{\pi_n}(s)})
    = 
    0.
\end{equation*}

By the discussion above, we can 
conclude that
\begin{gather*}
    \lim_{n \to \infty}
        (X_{\overline{\pi_n}(s) \wedge t} - X_{s})
        \otimes (X_{s} - X_{\underline{\pi_n}(s)})
    = 0
\end{gather*}
holds if $X$ is continuous at $s$.
\end{proof}

\begin{Lem} \label{4ec}
Suppose that $(\pi_n)$
satisfies Condition~(C) for $X \in D(\mathbb{R}_{\geq 0},E)$.
\begin{enumerate}
    \item If $X$ has strong (resp. weak) $B$-quadratic variation 
        along $(\pi_n)$, then
        \begin{equation} \label{4ed}
            Q_B(X,X)_t - Q_B(X,X)_s
        = 
            \lim_{n \to \infty}
                \sum_{\lrbrack r,u \rbrack \in \pi_n}
                B\left( 
                    X_{(u \wedge t) \vee s} - X_{(r \wedge t) \vee s}, 
                    X_{(u \wedge t) \vee s} - X_{(r \wedge t) \vee s}
                \right) 
        \end{equation}
        for all $t \geq s \geq 0$ in the norm (resp. weak) topology.
    \item If $X$ has scalar quadratic variation 
        along $(\pi_n)$, then $Q(X)$ satisfies
        \begin{equation} \label{4ee}
            Q(X)_t - Q(X)_s
        = 
            \lim_{n \to \infty}
                \sum_{\lrbrack r,u \rbrack \in \pi_n}
                \lVert X_{(u \wedge t) \vee s} - X_{(r \wedge t) \vee s} \rVert^2
        \end{equation}
        for all $s$ and $t$ with $s \leq t$.
\end{enumerate}
\end{Lem}

\begin{proof}
We first show (i).
Choose any two positive numbers $s$ and $t$ satisfying $s \leq t$.
By direct calculation, we see that
\begin{align*}
    Q_B^{\pi_n}(X,X)_t
    - Q_B^{\pi_n}(X,X)_s
    & = 
    \sum_{\lrbrack r,u \rbrack \in \pi_n}
        B(X_{(u \wedge t) \vee s} - X_{(r \wedge t) \vee s}, X_{(u \wedge t) \vee s} - X_{(r \wedge t) \vee s})  \\
    & \quad
        + B(X_{\overline{\pi_n}(s) \wedge t} - X_{s},X_{s} - X_{\underline{\pi_n}(s)})
        + B(X_{s} - X_{\underline{\pi_n}(s)},X_{\overline{\pi_n}(s) \wedge t} - X_{s}).
\end{align*}
This combined with Lemma~\ref{4eb} implies
\eqref{4ed} in the corresponding topology.

Next, we consider (ii).
Transforming the summation as 
\begin{align*}
    & 
        Q^{\pi_n}(X)_t - Q^{\pi_n}(X)_s  \\
    & \qquad = 
        \sum_{\lrbrack r,u \rbrack \in \pi_n}
            \lVert X_{(u \wedge t) \vee s} - X_{(r \wedge t) \vee s} \rVert^2
        - \lVert X_{\overline{\pi_n}(s) \wedge t} - X_{s} \rVert^2
        + \lVert \delta X_t(\pi_n(s)) \rVert^2
        - \lVert \delta X_s(\pi_n(s)) \rVert^2,
\end{align*}
we see that 
\begin{align*}
    & 
    \left\lvert
        Q^{\pi_n}(X)_t - Q^{\pi_n}(X)_s
        - \sum_{\lrbrack r,u \rbrack \in \pi_n}
            \lVert X_{(u \wedge t) \vee s} - X_{(r \wedge t) \vee s} \rVert^2
    \right\rvert  \\
    & \qquad \leq 
    \left\lvert 
        \lVert \delta X_t(\pi_n(s)) \rVert^2
        - \lVert X_{\overline{\pi_n}(s) \wedge t} - X_{s} \rVert^2
        - \lVert \delta X_s(\pi_n(s)) \rVert^2
    \right\rvert 
    \leq 
    2 \lVert X_{\overline{\pi_n}(s) \wedge t} - X_{s} \rVert
        \lVert X_{s} - X_{\underline{\pi_n}(s)} \rVert.
\end{align*}
Since the right-hand side converges to $0$ as $n \to \infty$ by Lemma~\ref{4eb}, 
we obtain \eqref{4ee}.
\end{proof}

\begin{proof}[Proof of Proposition~\ref{4d}]
Let $t \geq s \geq 0$ and
take an arbitrary $z^\ast \in G^\ast$ satisfying
$\lVert z^\ast \rVert \leq 1$.
Then
\begin{equation*}
    \left\lvert
        \left\langle z^\ast,\sum_{\lrbrack r,u \rbrack \in \pi_n}
            B\left( X_{(u \wedge t) \vee s} - X_{(r \wedge t) \vee s},X_{(u \wedge t) \vee s} - X_{(r \wedge t) \vee s} \right)
        \right\rangle
    \right\rvert
    \leq
    \lVert B \rVert
    \sum_{\lrbrack r,u \rbrack \in \pi_n}
        \left\lVert X_{(u \wedge t) \vee s} - X_{(r \wedge t) \vee s} \right\rVert^2,
\end{equation*}
and therefore, by Lemma~\ref{4ec}, we see that
\begin{equation*}
    \lvert \langle z^\ast,Q_B(X,X)_t - Q_B(X,X)_s \rangle \rvert
    \leq \lVert B \rVert (Q(X)_t - Q(X)_s).
\end{equation*}
By taking the supremum over all $z^\ast \in G^\ast$
with $\lVert z^\ast \rVert \leq 1$, we obtain \eqref{4e}.
\end{proof}

\begin{Rem} \label{4f}
The assumption that $Q_B(X,X)$ has finite variation is not used
in the proof of Proposition~\ref{4d}.
Therefore, if $X$ has scalar quadratic variation
and there is a path $Q_B(X,X) \in D(\mathbb{R}_{\geq 0}, E \widehat{\otimes} E)$
satisfying Condition~(i) of Definition~\ref{2b},
it necessarily satisfies \eqref{4e}.
In this case, $Q_B$ automatically has finite variation.
\end{Rem}

\begin{Cor} \label{4g}
Let $\alpha$ be a reasonable crossnorm on $E \otimes E$.
Suppose that $X \in D(\mathbb{R}_{\geq 0},E)$ has 
both scalar and weak $\alpha$-tensor quadratic variations 
along $(\pi_n)$ and $(\pi_n)$ satisfies Condition~(C) for $X$.
Then they satisfy
\begin{equation} \label{4h}
    \left\lVert {^\alpha [X,X]}_t - {^\alpha [X,X]}_s \right\rVert
    \leq
    Q(X)_t - Q(X)_s
\end{equation}
for all $s,t \in \mathbb{R}_{\geq 0}$ with $s \leq t$.
In particular, ${^\alpha [X,X]}$ is absolutely continuous with respect to $Q(X)$.
\end{Cor}

\begin{Th} \label{4i}
Suppose that $X \in D(\mathbb{R}_{\geq 0},E))$ 
has weak $B$-quadratic variation and scalar quadratic variation
along $(\pi_n)$ and $(\pi_n)$ satisfies (C) for $X$.
If $G$ has the RNP,
then there is a $q_B \in L^1_{\mathrm{loc}}(Q(X);G)$
such that 
\begin{equation*}
    Q_B(X,X)_t = \int_{\lrbrack 0,t \rbrack} q_B(s) \mathrm{d}Q(X)_s, \qquad \forall t \geq 0,
\end{equation*}
and $\lVert q_B(s) \rVert_E \leq \lVert B \rVert$ for 
$\mathrm{d}Q(X)$-almost every $s$.
Moreover, we have
\begin{equation} \label{4ib}
    V(Q_B(X,X))_t = \int_{\lrbrack 0,t \rbrack} \lVert q_B(s) \rVert \mathrm{d}Q(X)_s.
\end{equation}
\end{Th}

\begin{proof}
By Proposition~\ref{4d}, we know that
the path $Q_B(X,X)$ is absolutely continuous 
with respect to $Q(X)$.
Then
we can take a locally Bochner-integrable function
$q_B \colon \mathbb{R}_{\geq 0} \to G$
such that
\begin{equation*}
    Q_B(X,X)_t - Q_B(X,X)_s
    = \int_{\lrbrack s,t \rbrack} q_B(u) \,\mathrm{d}Q(X)_u
\end{equation*}
and 
\begin{equation} \label{4ic}
    \int_s^t \lVert q_B(u) \rVert \mathrm{d}\mathrm{Q}(X)_u
    =
    V(Q_B(X,X),[s,t])
    =
    V(Q_B(X,X))_t - V(Q_B(X,X))_s
\end{equation}
holds for all $s, t \in \mathbb{R}_{\geq 0}$ with $s \leq t$
(see Propositions~\ref{a1b} and~\ref{a1d}).
Equation~\eqref{4ic} directly implies \eqref{4ib}.
Moreover, \eqref{4ic} combined with estimate \eqref{4e}
shows that $\lVert q(u) \rVert \leq \lVert B \rVert$ holds
$\mathrm{d}Q(X)$-almost everywhere.
\end{proof}

\begin{Cor} \label{4j}
Let $\alpha$ be a reasonable crossnorm on $E \otimes E$.
Suppose that $X \in D(\mathbb{R}_{\geq 0},E)$
has weak $\alpha$-tensor quadratic variation and scalar quadratic variation
along $(\pi_n)$, and that $(\pi_n)$ satisfies (C) for $X$.
If $E \widehat{\otimes}_{\alpha} E$ has the RNP,
then there is a $q \in L^1_{\mathrm{loc}}(Q(X);E \widehat{\otimes} E)$
such that 
\begin{equation} \label{4jb}
    {^\alpha [X,X]}_t = \int_{\lrbrack 0,t \rbrack} q(s) \mathrm{d}Q(X)_s, \qquad
    V({^\alpha [X,X]})_t = \int_{\lrbrack 0,t \rbrack} \lVert q(s) \rVert \mathrm{d}Q(X)_s
\end{equation}
for all $t \geq 0$ and 
$\lVert q(s) \rVert_E \leq 1$ holds for $Q(X)$-almost every $s$.
\end{Cor}

\begin{Rem} \label{4jc}
    For the RNP of the projective tensor product of 
    Banach spaces, see Diestel, Fourie, and Swart~\cite{Diestel_Fourie_Swart_2006}
    and the references therein.
    If $E = L^p(\mu)$ with $1 < p < \infty$ on some measure space $(\Omega,\mathcal{A},\mu)$,
    then one can take a crossnorm $\alpha$ such that
    $L^p(\mu) \widehat{\otimes}_{\alpha} L^p(\mu)$ and
    $L^p(\mu \otimes \mu)$ are isomorphic Banach spaces.
    See Defant and Floret~\cite[Section 7]{Defant_Floret_1993} for details.
    In this case, $L^p(\mu) \widehat{\otimes}_{\alpha} L^p(\mu)$
    is reflexive and it has the RNP.
\end{Rem}

We can improve Corollary~\ref{4j} when the state space $E$ is a Hilbert space.

\begin{Cor} \label{4k}
Let $H$ be a separable Hilbert space
and $X \in D(\mathbb{R}_{\geq 0},H)$.
Suppose that $X$ has projective tensor quadratic variation
along $(\pi_n)$ and that $(\pi_n)$ satisfies (C) for $X$.
Then the density $q \in L^1_{\mathrm{loc}}(Q(X);H)$
in Proposition~\ref{4j} satisfies $\lVert q(s) \rVert = 1$
for $Q(X)$-almost every $s$.
Consequently, $V([X,X]) = Q(X)$. 
\end{Cor}

\begin{proof}
First recall that 
there is an isomorphism
$(H \widehat{\otimes}_{\varepsilon} H)^* \cong H 
\widehat{\otimes} H$,
where $\varepsilon$ denotes the injective tensor norm.
See, for example, Schatten~\cite[Theorem 5.13]{Schattenn_1950} or
Fabian~\cite[Proposition 16.40]{Fabian_Habala_Hajek_Santalucia_Zizler_2011}
for a proof.
Moreover, the separability of $H$ implies
that $H \widehat{\otimes} H$ is also separable.
Therefore, $H \widehat{\otimes} H$ has the RNP.
By Proposition~\ref{4j}, 
there is a function $q \in L^1_{\mathrm{loc}}(Q(X);H)$
satisfying \eqref{4jb} for all $t \geq 0$.
On the other hand, by Proposition~\ref{3g}, we have
\begin{equation*}
    Q(X)_t = \mathop{\mathrm{tr}}\nolimits_H([X,X]_t) 
    = 
    \int_{\lrbrack 0,t \rbrack} \mathop{\mathrm{tr}}\nolimits_H(q(s)) \mathrm{d}Q(X)_s
\end{equation*}
for all $t \geq 0$.
Since the trace functional
$\mathop{\mathrm{tr}}\nolimits_H \colon H \widehat{\otimes} H \to \mathbb{R}$
is contractive, we see that
$\lvert \mathop{\mathrm{tr}}\nolimits_H \circ q \rvert \leq \lVert q \rVert$
holds almost everywhere. Thus, we obtain the inequality
\begin{equation*}
    Q(X)_t \leq V([X,X])_t \leq Q(X)_t
\end{equation*}
for all $t \in \mathbb{R}_{\geq 0}$, which shows the assertion of the corollary.
\end{proof}

\section{\texorpdfstring{$C^1$}{C¹}-transformations}

In this section, we study quadratic variations of
a path of the form $f(t,X_t)$ defined by a c\`{a}dl\`{a}g path $X$ with quadratic variation
and a sufficiently nice function $f$.
First, let us recall classical results in the theory of classical It\^{o}'s stochastic calculus.
If $X$ is a semimartingale and $f$ is of class $C^{1,2}$,
then $f(\cdot,X_{\cdot})$ is still a semimartingale by the It\^{o} formula
and therefore it has quadratic variation.
This result can partially be extended to a $C^1$ function $f$ in the sense
that $f(\cdot,X_{\cdot})$ has quadratic variation
(see, e.g., Meyer~\cite[Theorem~5 in Chapter~VI]{Meyer_1976})
while it is not necessarily a semimartingale.
There are corresponding $C^1$-transformation results
in the It\^{o}--F\"{o}llmer calculus in Euclidean spaces
(see Sondermann~\cite{Sondermann_2006} and Hirai~\cite{Hirai_2019}).
We extend these previous results to infinite-dimensional paths.
As we treat {\cadlag} paths, it is natural to assume
that each path $t \mapsto f(t,x)$ is also {\cadlag}.
In this case, the roles of variables $t$ and $x$ are no longer symmetric,
and we regard $f$ as a {\cadlag} path $t \mapsto f(t,\cdot)$
in a space of functions with appropriate $C^1$-smoothness.

Note that Ananova and Cont~\cite{Ananova_Cont_2017} gives corresponding results
for more general path-dependent functionals $f$ in the case
where $X$ is finite-dimensional and continuous.
Extending our results to such path-dependent functionals is also important,
but beyond the scope of this article and thus not discussed here.

\subsection{Preliminaries}

In this subsection, we introduce some preliminary
concepts that will be used for the main results of this section.
First, we define a family of {\cadlag} paths
of uniformly finite variation.

\begin{Def} \label{5.1ab}
Let $E$ be a Banach space and $\mathcal{F}$ be
a nonempty subset of $D(\mathbb{R}_{\geq 0},E)$.
For each compact interval $[a,b] \subset \mathbb{R}_{\geq 0}$, define
\begin{equation*}
    V(\mathcal{F};[a,b])
    \coloneqq \sup_{f \in \mathcal{F}} V(f;[a,b])
    = \sup_{f \in \mathcal{F}} \sup_{\pi \in \Par([a,b])}
        \sum_{I \in \pi} \lVert \delta f(I) \rVert.
\end{equation*}
We say that $\mathcal{F}$ has \emph{uniformly finite variation}
if $V(\mathcal{F};[a,b]) < \infty$
for all compact intervals $[a,b] \subset \mathbb{R}_{\geq 0}$.
A parametrized family $(f_\lambda)_{\lambda \in \Lambda}$ of elements of 
$D(\mathbb{R}_{\geq 0},E)$ has uniformly finite variation
if the set $\{ f_i \mid i \in I \}$ does.
\end{Def}

By definition,
$\mathcal{F}$ has uniformly finite variation
if and only if $(V(f;I) ; f \in \mathcal{F})$ is 
bounded for all compact intervals $I \subset \mathbb{R}_{\geq 0}$.
We simply write $V(\mathcal{F})_t = V(\mathcal{F};[0,t])$ for $t \geq 0$.
Given a parametrized family $(f_{\lambda})_{\lambda \in \Lambda}$
of {\cadlag} paths,
set $V(f_\lambda; \lambda \in \Lambda)_t = V(\{ f_{\lambda};\lambda \in \Lambda \})_t$.

Next, we introduce a variant of Definition~\ref{2d}
for a family of c\`{a}dl\`{a}g paths.
This condition will be used to formulate the main
results of this section, such as
Theorem~\ref{5.2b} and Corollary~\ref{5.2kb}.
Given a subset $\mathcal{F}$ of $D(\mathbb{R}_{\geq 0},E)$,
define
\begin{gather*}
D(\mathcal{F})
    =
    \{
        s \in [0,\infty \rlbrack \mid \text{$\lVert \Delta f \rVert_{s} > 0 $ holds for some $f \in \mathcal{F}$}
    \},   \\
D_{\varepsilon}(\mathcal{F})
    =
    \{ 
        s \in [0,\infty \rlbrack 
        \mid 
        \text{$\lVert \Delta f \rVert_{s} > \varepsilon $ holds for some $f \in \mathcal{F}$}
    \},    \\
D^{\varepsilon}(\mathcal{F})
    = D(\mathcal{F}) \setminus D_{\varepsilon}(\mathcal{F}).
\end{gather*}
Note that $D_{\varepsilon}(\mathcal{F})$
can be an uncountable set in general, 
but, as a consequence of Proposition~\ref{a2.1d},
it is countable for every $\varepsilon > 0$
provided that $\mathcal{F}$ is equi-right-regular,
as defined in Appendix~B.

\begin{Def} \label{5.1b}
Let $\mathcal{F}$ be a subset of $D(\mathbb{R}_{\geq 0},E)$
and $(\pi_n)$ be a sequence of partitions of $\mathbb{R}_{\geq 0}$.
We say that $(\pi_n)$ satisfies (UC) for $\mathcal{F}$
if it satisfies the following three conditions:
\begin{enumerate}
    \item[(UC1)] For every $\varepsilon$ and every $t > 0$, 
        there exists an $N$ such that 
        $I \cap D_{\varepsilon}(\mathcal{F}) \cap [0,t]$ has 
        at most one element for all $n \geq N$ and $I \in \pi_n$.
    \item[(UC2)] For all $s \in D(\mathcal{F})$ and $t \geq s$,
    the sequence $(\delta f_t(\pi_n(s)))_{n \in \mathbb{N}}$
    converges to $\Delta f(s)$ uniformly in $f \in \mathcal{F}$.
    \item[(UC3)] For all $t \in [0,\infty \rlbrack$
        \begin{equation*}
            \varlimsup_{\varepsilon \downarrow\downarrow 0 } \varlimsup_{n \to \infty}
            \sup_{f \in \mathcal{F}} O_t^+(f-J_{\varepsilon}(f),\pi_n)  = 0.
        \end{equation*}
\end{enumerate}
\end{Def}

Although the assumption (UC) seems strong,
we can always take a sequence of partitions $(\pi_n)$ satisfying (UC)
if $\mathcal{F}$ is equi-right-regular
(see Definition~\ref{a2.1b}),
as in the following example.

\begin{Exm} \label{5.1c}
Let $E$ be a Banach space,
$\mathcal{F}$ be a subset of $D(\mathbb{R}_{\geq 0},E)$,
and $(\pi_n)$ be a sequence of partitions of $\mathbb{R}_{\geq 0}$. 
\begin{enumerate}
    \item If $\mathcal{F}$ is equi-right-regular and 
        $\lvert \pi_n \rvert \to 0$ as $n \to \infty$,
        then $(\pi_n)$ satisfies (UC) for $\mathcal{F}$.
        This is a consequence of Proposition~\ref{a2.1d}.
    \item For each $t \geq 0$ and $n \in \mathbb{N}$, define
        \begin{equation*}
            O_t^{-}(\mathcal{F},\pi_n)
            =
            \sup_{f \in \mathcal{F}} \sup_{\lrbrack r,s \rbrack \in \pi_n} \omega(f,\lbrack r,s \rlbrack \cap [0,t]).
        \end{equation*}
        If $O_t^-(\mathcal{F},\pi_n) \to 0$ as $n \to \infty$ for every $t \geq 0$,
        then $(\pi_n)$ satisfies (UC) for $\mathcal{F}$.
        One can always select such a sequence $(\pi_n)$ whenever $\mathcal{F}$ is equi-right-regular.
\end{enumerate}
\end{Exm}

Let us now introduce the function spaces
used to formulate the $C^1$-transformation formulae
in the next subsection.

\begin{Def} \label{5.1d}
\begin{enumerate}
    \item Let $T$ and $S$ be topological spaces.
        We say that a function $f \colon T \to S$
        belongs to $C_{\mathcal{K}}(T,S)$
        if its restriction $f\vert_K \colon K \to S$
        is continuous for each compact topological subspace $K$ of $T$.
    \item Let $E$ and $F$ be Banach spaces.
        We define $C^1_{\mathcal{K}}(E,F)$ to be the set of 
        all functions $f \in C_{\mathcal{K}}(E,F)$ satisfying the following conditions: 
        \begin{enumerate}
            \item the function $f$ is G\^{a}teaux differentiable;
            \item the restriction of the G\^{a}teaux derivative
                $D_x f\vert_{K} \colon K \to \mathcal{L}(E,F)$
                is continuous for each compact subset $K$ of $E$.
        \end{enumerate}
\end{enumerate}
\end{Def}

Here note that conditions~(a) and (b) themselves
imply that $f \in C_{\mathcal{K}}(E,F)$.
Therefore we can simply restate that $C^1_{\mathcal{K}}(E,F)$
is the set of all G\^{a}teaux differentiable functions
with derivatives in $C_{\mathcal{K}}(E,\mathcal{L}(E,F))$.

As usual, we regard $C_{\mathcal{K}}(E,F)$ as
a locally convex Hausdorff topological vector space
with the topology of uniform convergence on compact subsets.
The topology of $C_{\mathcal{K}}(E,F)$ is
generated by the family of seminorms
$(\lVert \phantom{x} \rVert_{\infty,K})_{K}$ defined by
\begin{equation*}
    \lVert f \rVert_{\infty,K} \coloneqq
    \sup_{x \in K} \lVert f(x) \rVert_F,
\end{equation*}
where $K$ runs over all compact subsets of $E$.
Similarly, we define a topology of $C^1_{\mathcal{K}}(E,F)$
using the seminorms
\begin{equation*}
    \lVert f \rVert_{C^1_{\mathcal{K}},K}\
    \coloneqq \sup_{x \in K} \lVert f(x) \rVert_F + \sup_{x \in K} \lVert Df(x) \rVert_{\mathcal{L}(E,F)}
\end{equation*}
indexed by all compact subsets $K$ of $E$.

Recall Ascoli's theorem,
which characterizes the total boundedness of a subset of $C_{\mathcal{K}}(E,F)$.
Refer to Bourbaki~\cite[X.2.5 Theorem 2]{Bourbaki_1966b}
for a proof.
To state the theorem,
we introduce the notion of uniform equicontinuity.
Let $A$ be a subset of $E$.
We say that $\mathcal{F} \subset C(A,F)$ is \emph{uniformly equicontinuous} if,
for all $\varepsilon > 0$,
there exists a $\delta > 0$ such that
$\lVert g(x) - g(y) \rVert_F < \varepsilon$
holds for all $g \in \mathcal{F}$ and $x,y \in A$ with $\lVert x-y \rVert < \delta$.

\begin{Th}[Ascoli] \label{5.1e}
A subset $\mathcal{F}$ of $C_{\mathcal{K}}(E,F)$
is totally bounded if and only if it satisfies the following conditions:
\begin{enumerate}
    \item the set $\{ f(x) \mid f \in \mathcal{F} \}$ is totally bounded in $F$ for each $x \in E$;
    \item the set $\{ f\vert_K \mid f \in \mathcal{F}\}$ is uniformly equicontinuous for each compact set $K \subset E$.
\end{enumerate}
\end{Th}

For two Banach spaces $E$ and $F$,
we can consider the space $D(\mathbb{R}_{\geq 0},C_{\mathcal{K}}(E,F))$
of $C_{\mathcal{K}}(E,F)$-valued {\cadlag} paths.
Note that each $f \in D(\mathbb{R}_{\geq 0},C_{\mathcal{K}}(E,F))$
can be regarded as a function on $\mathbb{R}_{\geq 0} \times E$,
so we use two notations, $f_t(x)$ and $f(t,x)$,
for the value of $f$ in $F$.
In the remainder of this subsection,
we study some properties of the path $t \mapsto f(t,X_t)$
defined by an $f \in D(\mathbb{R}_{\geq 0},C_{\mathcal{K}}(E,F))$
and an $X \in D(E,F)$.
These properties will be used in the proof of
the main theorem of Section~5.2.

We begin with the following preliminary theorem
from the theory of general topology.
See, for example, Kelley~\cite[7.6 and 7.10 (e)]{Kelley_1975}
and Bourbaki~\cite[Remark X.2.5]{Bourbaki_1966b}.

\begin{Lem} \label{5.1f}
Let $T$ be a Hausdorff topological space
and $S$ be a metric space (or, more generally, a uniform space).
Define the evaluation map 
$\operatorname{ev} \colon T \times C_{\mathcal{K}}(T,S) \to S$
by the formula $\mathrm{ev}(x,f) = f(x)$.
We regard $C_{\mathcal{K}}(T,S)$ as a topological space 
endowed with the topology of uniform convergence on compact subsets.
\begin{enumerate}
    \item The restriction $\operatorname{ev}\vert_{K \times C_{\mathcal{K}}(T,S)}$
        is jointly continuous for each compact subset $K$ of $T$.
    \item Let $A$ be a compact subset of $C_{\mathcal{K}}(T,S)$.
        Then the restriction $\operatorname{ev}\vert_{K \times A}$
        is jointly continuous relative to the topology of pointwise convergence on $A$
        for each compact subset $K$ of $T$.
\end{enumerate}
\end{Lem}

\begin{Lem} \label{5.1g}
If $f \in D(\mathbb{R}_{\geq 0},C_{\mathcal{K}}(E,F))$ and $X \in D(\mathbb{R}_{\geq 0},E)$,
then the path $t \mapsto f(t,X_t)$ is {\cadlag} and the left limits are given by
\begin{equation*}
    \lim_{s \uparrow\uparrow t} f(s,X_s) = f(t-,X_{t-}), \qquad t \geq 0.
\end{equation*}
\end{Lem}

\begin{proof}
To show the assertion, 
it suffices to verify that $t \mapsto f(t,X_t)$ is {\cadlag} on
every compact interval of the form $[0,T]$.
Fix $T > 0$ arbitrarily.
Since the image $X([0,T])$ is totally bounded in $E$,
one can regard the function
$f(\cdot,X_{\cdot}) \colon [0,T] \to F$
as the composition of $(X,f)\vert_{[0,T]}$ and
$\mathrm{ev}\vert_{K \times C_{\mathcal{K}}(E,F)}$
for some compact set $K \subset E$.
The restricted evaluation map
$\mathrm{ev}\vert_{K \times C_{\mathcal{K}}(E,F)}$ is continuous by Lemma~\ref{5.1f},
and therefore 
the composition $f(\cdot,X_{\cdot})\vert_{[0,T]} = \mathrm{ev}\vert_{K \times C_{\mathcal{K}}(E,F)} \circ (X,f)\vert_{[0,T]}$
is {\cadlag}.
Moreover, again by the continuity of the evaluation map,
we obtain
\begin{equation*}
    \lim_{s \uparrow\uparrow t} f(s,X_s)
    =
    \mathrm{ev}\vert_{K \times C_{\mathcal{K}}(E,F)} (X_{s-},f_{s-})
    =
    f(s-,X_{s-}).
\end{equation*}
This completes the proof.
\end{proof}

\begin{Lem} \label{5.1h}
Let $(X,f) \in D(\mathbb{R}_{\geq 0},E \times C_{\mathcal{K}}(E,F))$,
$T > 0$, and $K$ be a compact subset satisfying $X([0,T]) \subset K$.
Assume that a sequence of partitions $(\pi_n)$ approximate
$X$ from the left on $[0,T]$ and so does $f(\cdot,x)$ for all $x \in K$.
Then $(\pi_n)$ approximate $f(\cdot,X_{\cdot})$ from the left on $[0,T]$.
\end{Lem}

\begin{proof}
First, note that the restriction map
$C_{\mathcal{K}}(E,F) \to C_{\mathcal{K}}(K,F)$
is continuous because $K$ is compact.
We take a compact set $C \subset C_{\mathcal{K}}(E,F)$
that includes the image $f([0,T])$.
Then $C\vert_K \coloneqq \{ f\vert_K \mid f \in C \}$
is compact in $C_{\mathcal{K}}(K,F)$
because it is the image of compact set $C$ by the restriction map
$C_{\mathcal{K}}(E,F) \to C_{\mathcal{K}}(K,F)$.
Therefore $\mathrm{ev}\vert_{K \times C\vert_K}$
is continuous relative to the topology of pointwise convergence
on $C_{\mathcal{K}}(K,F)$ by Lemma~\ref{5.1f}.
This continuity combined with the assumptions on
$(\pi_n)$, $X$, and $f$ implies that
\begin{align*}
    \lim_{n \to \infty} f(\underline{\pi_n}(s),X_{\underline{\pi_n}(s)})
    =
    \mathrm{ev}\vert_{K \times C\vert_K} (X_{s-},f_{s-}\vert_K) 
    =
    f(s-,X_{s-}).
\end{align*}
This shows the assertion. 
\end{proof}

\subsection{Main results of the section}

In this subsection, 
let $E$, $E_1$, $F$, $F_1$, and $G$ be Banach spaces.
Given $B \in \mathcal{L}^{(2)}(E,E;E_1)$ and $B' \in \mathcal{L}^{(2)}(F,F;F_1)$,
we identify $B$ and $B'$ with elements of $\mathcal{L}(E \widehat{\otimes} E,E_1)$
and $\mathcal{L}(F \widehat{\otimes} F,F_1)$, respectively, by canonical isomorphisms.
We also assume that there is a given sequence of partitions of $\mathbb{R}_{\geq 0}$,
denoted by $(\pi_n)$.

Let us now give the first $C^1$-transformation formula.
The following theorem, which seems somewhat complicated, 
provides a sufficient condition 
for the existence of $B'$-quadratic variation
of a path $f(\cdot,X_{\cdot})$.
It will later be used to derive the $C^1$-transformation formula
for tensor quadratic variations.

\begin{Th} \label{5.2b}
Let $X \colon \mathbb{R}_{\geq 0} \to E$ be a 
c\`{a}dl\`{a}g path that has weak $B$-quadratic variation
and finite 2-variation along $(\pi_n)$,
and let $f \colon \mathbb{R}_{\geq 0} \to C^1_{\mathcal{K}}(E,F)$
a {\cadlag} path such that the family
$(f(\cdot,x); x \in K)$ has uniformly finite variation for every
compact set $K \subset E$.
Assume that there is a {\cadlag} function
$\Phi_f \colon \mathbb{R}_{\geq 0} \to C_{\mathcal{K}}(E,\mathcal{L}(E_1,F_1))$
that commutes the following diagram for all $(t,x) \in \mathbb{R}_{\geq 0} \times E$.
\begin{equation*}
    \begin{tikzcd}[column sep=large]
        E \widehat{\otimes} E \arrow[r,"{D_xf(t,x)}^{\otimes 2}"] \arrow[d,"B"']
        & F \widehat{\otimes} F \arrow[d,"B'"] \\
        E_1 \arrow[r,"{\Phi_f(t,x)}"']
        & F_1
    \end{tikzcd}
\end{equation*}
Let $T > 0$ and suppose that there is a compact convex set $K \subset E$
satisfying the following conditions:
\begin{enumerate}
    \item the image $X([0,T])$ is included in $K$;
    \item the sequence $(\pi_n)$ approximates $(X,f(\cdot,x),\Phi_f(\cdot,x))$ from the left for all $x \in K$;
    \item the sequence $(\pi_n)$ satisfies (UC) for the family
        of $E \times F$-valued {\cadlag} paths $(X,f_{\cdot}(x))_{x \in K}$.
\end{enumerate}
Under these assumptions, the path $[0,T] \ni t \mapsto f(t,X_t) \in F$ has
the weak $B$-quadratic variation
\begin{equation} \label{5.2c}
    Q_{B'}(f(\cdot,X_{\cdot}),f(\cdot,X_{\cdot}))_t
    =
    \int_0^t
            \Phi_f (s-,X_{s-})
            \mathrm{d}Q_B(X,X)^{\mathrm{c}}_s
    + \sum_{0 < s \leq t} B'(\Delta f(s,X_s))^{\otimes 2}
\end{equation}
for $t \in [0,T]$.
If we assume, moreover, that $X$ has strong $B$-quadratic variation,
then the path given by \eqref{5.2c} is
the strong $B'$-quadratic variation of $f(\cdot,X_{\cdot})$. 
\end{Th}

\begin{Lem} \label{5.2d}
Under the assumptions of Theorem~\ref{5.2b},
$(\pi_n)$ satisfies (C1) and (C2) of Definition~\ref{2d} for $f(\cdot,X_{\cdot})$ on $[0,T]$.
\end{Lem}

\begin{proof}
Assume that $T > 0$ and a compact set $K \subset E$ satisfies (i)--(iii)
of Theorem~\ref{5.2b}.
First, we check the following estimate
\begin{align*}
    \lVert \Delta f(s,X_s) \rVert
    & \leq
    \sup_{x \in K} \lVert f(s,x) -f(s-,x) \rVert
    +
    \int_0^1 \lVert Df(s-,X_{s-} + \theta \Delta X_s)\Delta X_s \rVert \mathrm{d}\theta   \\
    & \leq
    \lVert \Delta f_s \rVert_{\infty,K} + \sup_{(r,x) \in [0,T] \times K} \lVert Df(r,x) \rVert \lVert \Delta X_s \rVert.
\end{align*}
Then, by setting $C = 1 \vee \sup_{(r,x) \in [0,T] \times K} \lVert D_x f(r,x) \rVert$,
we have 
\begin{equation*}
    \lVert \Delta f(s,X_s) \rVert
    \leq
    C (\lVert \Delta X_s \rVert + \lVert \Delta f(s) \rVert_{\infty,K}).
\end{equation*}
This shows
$
    D_{\varepsilon}(f(\cdot,X_{\cdot})) 
    \subset 
    D_{C^{-1}\varepsilon}((X,f(\cdot,x));x \in K)
$, which implies Condition (C1).

To show (C2), let $s \in D(f(\cdot,X_{\cdot})) \cap \lrbrack 0, T \rbrack$ and $t \in [s,T]$.
The discussion in the previous paragraph implies
$
    D(f(\cdot,X_{\cdot}))
    \subset
    D((X,f),\lVert \phantom{x} \rVert + \lVert \phantom{x} \rVert_{\infty,K})
$. Therefore, by assumption,
\begin{align*}
    \lim_{n \to \infty} X(\overline{\pi_n}(s) \wedge t) & = X_s,
    & \lim_{n \to \infty} X(\underline{\pi_n}(s) \wedge t) & = X_{s-},  \\
    \lim_{n \to \infty} f(\overline{\pi_n}(s) \wedge t,x) & = f(s,x),
    & \lim_{n \to \infty} f(\underline{\pi_n}(s),x) & = f(s-,x)
\end{align*}
holds for all $x \in K$.
This combined with Lemma~\ref{5.1g} implies
\begin{align*}
    \lim_{n \to \infty} 
        \{ 
            f(\overline{\pi_n}(s) \wedge t, X_{\overline{\pi_n}(s) \wedge t}) 
            -f(\underline{\pi_n}(s) \wedge t, X_{\underline{\pi_n}(s) \wedge t})
        \} 
    & =
    f(s,X_s) - f(s-,X_{s-}),
\end{align*}
which completes the proof.
\end{proof}

\begin{proof}[Proof of Theorem~\ref{5.2b}]
Let $K$ be a compact convex set satisfying (i)--(iii) of the theorem.
For convenience, let
\begin{equation*}
    \mathcal{F}_K = \{ f(\cdot,x) \mid x \in K\} \subset D(\mathbb{R}_{\geq 0},F),  \qquad
    \mathcal{F}_{X,K} = \{ (X,f(\cdot,x)) \mid x \in K \} \subset D(\mathbb{R}_{\geq 0},E \times F).
\end{equation*}

\emph{Step 1: Convergence of jumps.}
First, we check the absolute convergence of the jump part
of \eqref{5.2c}.
Observing the estimate
\begin{align*}
    \sum_{0 < s \leq T} \lVert \Delta f(s,X_s)^{\otimes 2} \rVert
    & \leq 
    2 \sup_{(s,x) \in [0,T] \times K} \lVert D_x f(s,x) \rVert^2
        \sum_{0 < s \leq T} \lVert \Delta X_s \rVert^2
    + 2 \sup_{(s,x) \in [0,T] \times K} \lVert f(s,x) \rVert
        V(\mathcal{F}_K)_t,
\end{align*}
we see that $B(\Delta f(\cdot,X_{\cdot})^{\otimes 2})$
is absolutely summable on $[0,T]$.
Note that
\begin{gather*}
    \sup_{(s,x) \in [0,t] \times K} \lVert f(s,x) \rVert
    =
    \sup_{s \in [0,t]} \lVert f_s \rVert_{\infty,K} < \infty,   \\
    \sup_{(s,x) \in [0,t] \times K} \lVert D_x f(s,x) \rVert
    =
    \sup_{s \in [0,t]} \lVert D_x f(s,\cdot) \rVert_{\infty,K}
    <
    \infty
\end{gather*}
because $f$ is a $C^1_{\mathcal{K}}(E,F)$-valued {\cadlag} path.
Moreover, we have
\begin{equation*} 
    \sum_{0 < s \leq t} 
        \left\lVert 
            \Phi_f(s-,X_{s-}) B(\Delta X_s,\Delta X_s)
        \right\rVert 
    \leq 
    \lVert B' \rVert \sup_{(s,x) \in [0,t] \times K} \lVert D_x f(s,x) \rVert^2
        \sum_{0 < s \leq t} \lVert \Delta X_s \rVert^2 
    < 
    \infty.
\end{equation*}
Therefore, Equation~\eqref{5.2c} is equivalent to 
\begin{align} \label{5.2e}
    & 
    Q_{B'}(f_{\cdot}(X_\cdot),f_{\cdot}(X_\cdot))_t  \\
    & \qquad =
    \int_0^t
        \Phi_f(s-,X_{s-})\mathrm{d}Q_B(X,X)_s
    + \sum_{0 < s \leq t} B'(\Delta f(s,X_s)^{\otimes 2})
    - \sum_{0 < s \leq t} \Phi_f(s-,X_{s-})B(\Delta X_s,\Delta X_s).   \notag
\end{align}
We shall show \eqref{5.2e} instead of \eqref{5.2c} in the rest of this proof.

\emph{Step 2: The Taylor expansion.}
Let $I = \lrbrack r,s \rbrack \in \pi_n$.
Then by the first order Taylor expansion, we obtain
\begin{align} \label{5.2f}
    \delta f(\cdot,X_{\cdot})_t(I) 
    & =
    f(s \wedge t, X_{s \wedge t}) - f(r \wedge t,X_{s \wedge t})
    + D_x f(\cdot,X_{\cdot})\delta X_t(I)
    + R_t(I),
\end{align}
where $R_t(I)$ is defined by
\begin{equation*}
    R_t(I)
    =
    \int_{[0,1]} 
        \{ D_x f(r \wedge t,X_{r \wedge t} + \theta \delta X_t(I)) - D_xf(r \wedge t,X_{r \wedge t}) \} \delta X_t(I)
    \mathrm{d}\theta.
\end{equation*}
For notational convenience, let 
\begin{equation*}
    \delta'_{f,X}(I) 
    = \mathrm{ev}(\delta f_t(I),X_{s \wedge t})
    = f(s \wedge t, X_{s \wedge t}) - f(r \wedge t,X_{s \wedge t}).
\end{equation*}
By \eqref{5.2f} and bilinearity, we see that
\begin{align} \label{5.2g}
    B'(\delta f(\cdot,X_{\cdot})_t^{\otimes 2})  
    =
    B'({\delta'}_{f,X}^{\otimes 2})(I) 
    + \Phi_f(\cdot,X_{\cdot})B(\delta X_t^{\otimes 2})(I)
    + B'(R_t^{\otimes 2})(I)
    + T^1_t(I) 
    + T^2_t(I)
    + T^3_t(I),
\end{align}
where
\begin{align*}
    T^1_t(I)
    & =
    B'\bigl[
        \delta'_{f,X}(I) \otimes D_x f(\cdot,X_\cdot) \delta X_t(I)
        + D_x f(\cdot,X_\cdot) \delta X_t(I) \otimes \delta'_{f,X}(I)
    \bigr]   \\
    T^2_t(I)
    & =
    B' \bigl[
        \delta'_{f,X}(I) \otimes R_t(I)
        + R_t(I) \otimes \delta'_{f,X}(I)
    \bigr]  \\
    T^3_t(I)
    & =
    B' \bigl[
        D_x f(\cdot,X_{\cdot})\delta X_t(I) \otimes R_t(I) + R_t(I) \otimes D_x f(\cdot,X_{\cdot}) \delta X_t(I)
    \bigr].
\end{align*}
Here we introduce the notation
\begin{equation*}
    e^1_{E}(A) = 
    \begin{cases}
        1 & \text{if} \ E \cap A \neq \emptyset, \\
        0 & \text{if} \ E \cap A = \emptyset,
    \end{cases}
    \qquad 
    e^2_E = 1-e^1_E.
\end{equation*}
Moreover, let $D = D(\mathcal{F}_{X,K})$,
$D_{\varepsilon} = D_{\varepsilon}(\mathcal{F}_{X,K})$, and $D^{\varepsilon} \setminus D_{\varepsilon}$
for each $\varepsilon > 0$.
Notice that $D = \bigcup_{\varepsilon > 0} D_\varepsilon$.
Then we can deduce from \eqref{5.2g} that 
\begin{align} \label{5.2h}
    & B'(\delta f(\cdot,X_{\cdot})_t^{\otimes 2})
    - e^1_{D_{\varepsilon}} B'(\delta f(\cdot,X_{\cdot})_t^{\otimes 2})    \\
    & \qquad =
    \Phi_f(\cdot,X_{\cdot})B(\delta X_t^{\otimes 2}) - e^1_{D_{\varepsilon}} \Phi_f(\cdot,X_{\cdot})B(\delta X_t^{\otimes 2})  \notag \\
    & \qquad \quad
    + e^2_{D_{\varepsilon}} B'({\delta'}_{f,X}^{\otimes 2})
    + e^2_{D_{\varepsilon}} B'(R_t^{\otimes 2})
    + e^2_{D_{\varepsilon}} T^1_t
    + e^2_{D_{\varepsilon}} T^2_t
    + e^2_{D_{\varepsilon}} T^3_t   \notag
\end{align}
Summing each term of \eqref{5.2h} over the partition $\pi_n$, we obtain
\begin{align} \label{5.2i}
    & 
    \sum_{I \in \pi_n} B'(\delta f(\cdot,X_{\cdot})_t^{\otimes 2})(I)   \\
    & \qquad =
    \sum_{I \in \pi_n}
        e^1_{D_{\varepsilon}}(I) B'(\delta f(\cdot,X_{\cdot})_t^{\otimes 2})(I) 
    + \sum_{I \in \pi_n}
        e^2_{D_{\varepsilon}} (I)
        B'({\delta'}_{f,X}^{\otimes 2})(I)     \notag \\
    & \qquad \quad 
    + \sum_{I \in \pi_n}
        \Phi_f(\cdot,X_{\cdot}) B(\delta X_t^{\otimes 2})(I)
    - \sum_{I \in \pi_n}
        e^1_{D_{\varepsilon}} (I)
        \Phi_f(\cdot,X_{\cdot}) B(\delta X_t^{\otimes 2})(I)    \notag \\
    & \qquad \quad 
    + \sum_{I \in \pi_n}
        e^2_{D_{\varepsilon}} (I)
        B'(R_t(I)^{\otimes 2}) 
    + \sum_{I \in \pi_n} 
        e^2_{D_{\varepsilon}} (I) T^1_t(I)
    + \sum_{I \in \pi_n} 
        e^2_{D_{\varepsilon}} (I) T^2_t(I)
    + \sum_{I \in \pi_n} 
        e^2_{D_{\varepsilon}} (I) T^3_t(I)   \notag \\
    & \qquad \eqcolon
    I_1^{(n)} + I_2^{(n)} + I_3^{(n)} - I_4^{(n)}
    +I_5^{(n)} + I_6^{(n)} + I_7^{(n)} + I_8^{(n)}.  \notag
\end{align}
Therefore, it suffices to observe the limit of each $I_i^{(n)}$.

\emph{Step 3: Behavior of $I_i^{(n)}$'s.}
We first treat $i \in \{ 1,3,4 \}$.
By Lemma~\ref{5.1h} and assumptions (C1) and (C2) for $X$, 
we see that
\begin{align*}
    \lim_{n \to \infty} I^{(n)}_4
    = \lim_{n \to \infty} \sum_{s \in D_{\varepsilon} \cap [0,t]} 
    \Phi_f(\underline{\pi_n}(s),X_{\underline{\pi_n}(s)}) B (\delta X_t^{\otimes 2})(\pi_n(s))
    =
    \sum_{s \in D_{\varepsilon} \cap [0,t]} \Phi_f(s-,X_{s-}) B(\Delta X_s^{\otimes 2}).
\end{align*}
By Lemma~\ref{5.2d} and the assumption, $(\pi_n)$ and $f(\cdot,X_{\cdot})$
satisfies (C1) and (C2) on $[0,t]$. Therefore, we can deduce that
\begin{equation*}
    \lim_{n \to \infty} I^{(n)}_1 = \sum_{s \in D_{\varepsilon} \cap [0,t]} \Delta f(s,X_s).
\end{equation*}
Lemma~\ref{2g} implies that
\begin{equation*}
    \lim_{n \to \infty} I_3^{(n)}
    = \int_{\lrbrack 0,t \rbrack} \Phi_f(s-,X_{s-})\, \mathrm{d}Q_B(X,X)_s
\end{equation*}
holds in the weak topology.

It remains to observe the behaviour of residual terms $I^{(n)}_i$'s for $i \in \{ 2,5,6,7,8\}$,
For convenience, set
\begin{align*}
    \alpha(\varepsilon,n) & = O^+_t(X-J_{\varepsilon}(X);\pi_n),   \\
    \beta(\varepsilon,n) & = \sup_{x \in K} O_t^+(f_\cdot(x)-J_{\varepsilon}(f_{\cdot}(x));\pi_n),   \\
    C & = \sup_{(s,x) \in [0,t] \times K} \lVert D_x f(s,x) \rVert.
\end{align*}
Then we notice that the estimates
\begin{equation*}
    e^{2}_{D_{\varepsilon}}(I) \lVert \delta X_t(I) \rVert
    \leq
    \alpha(\varepsilon,n), \qquad
    e^{2}_{D_{\varepsilon}}(I) \lVert \delta f_t(I) \rVert_{\infty,K}
    \leq
    \beta(\varepsilon,n)
\end{equation*}
hold for $I \in \pi_n$.
This leads to the inequalities 
\begin{align*}
    \lVert I_2^{(n)} \rVert
    & \leq
    \lVert B' \rVert V(\mathcal{F}_K)_t
    \beta(\varepsilon,n),   \\
    \lVert I^{(n)}_5 \rVert
    & \leq 
    V^{(2)}(X;\Pi)_t
        \sup_{
            \substack{
                s \in [0,t], x,y \in K  \\
                \lVert x - y \rVert \leq \alpha(\varepsilon,n)
            }
        }
        \lVert D_xf(s,x) - D_x f(s,y) \rVert^2,   \\
    \lVert I^{(n)}_6 \rVert
    & \leq 
    2 C \lVert B' \rVert V(\mathcal{F}_K)_t \alpha(\varepsilon,n),   \\
    \lVert I^{(n)}_7 \rVert
    & \leq
    4 C \lVert B' \rVert V(\mathcal{F}_K)_t \alpha(\varepsilon,n),  \\
    \lVert I^{(n)}_8 \rVert 
    & \leq 
    2 C \lVert B' \rVert 
        V^{(2)}(X;\Pi)_t
        \sup_{
            \substack{
                s \in [0,t], x,y \in K  \\
                \lVert x - y \rVert \leq \alpha(\varepsilon,n)
            }
        }
        \lVert D_xf(s,x) - D_x f(s,y) \rVert.
\end{align*}
Therefore,
\begin{equation*}
    \varlimsup_{\varepsilon \downarrow\downarrow 0} \varlimsup_{n \to \infty} I^{(n)}_i = 0, \qquad 
    i \in \{2,5,6,7,8 \}.
\end{equation*}
Note that here we have used
\begin{equation} \label{5.2ib}
    \varlimsup_{\varepsilon \downarrow\downarrow 0} \varlimsup_{n \to \infty}
    \sup_{
            \substack{
                s \in [0,t], x,y \in K  \\
                \lVert x - y \rVert \leq \alpha(\varepsilon,n)
            }
        }
        \lVert D_xf(s,x) - D_x f(s,y) \rVert
    = 0.
\end{equation}
To deduce \eqref{5.2ib}, recall that
$D_x f \colon \mathbb{R}_{\geq 0} \to C_\mathcal{K}(E,\mathcal{L}(E,F))$
is {\cadlag} by assumption
and therefore it has totally bounded range on $[0,T]$.
Hence, by Ascoli's theorem (Theorem~\ref{5.1e}),
the family $(D_x f_s\vert_{K};s \in [0,t])$ is uniformly equicontinuous,
which guarantees \eqref{5.2ib}.

\emph{Step 4: Conclusion.}
Combining all the estimates obtained in step~3,
we obtain for all $z^\ast \in G^\ast$ with $\lVert z^\ast \rVert = 1$
\begin{align} \label{5.2ic}
    &
    \varlimsup_{n \to \infty}
    \left\lvert
        \langle 
            z^\ast, 
            (\text{RHS of \eqref{5.2e}}) - (\text{RHS of \eqref{5.2i}}) 
        \rangle 
    \right\rvert   \\
    & \qquad \leq
    \varlimsup_{\varepsilon \downarrow \downarrow 0} \varlimsup_{n \to \infty}
        \lVert B' \rVert \sum_{s \in D^{\varepsilon} \cap [0,t]} \lVert \Delta f(s,X_s) \rVert^2 
    + \varlimsup_{\varepsilon \downarrow \downarrow 0} \varlimsup_{n \to \infty}
        \sum_{s \in D^{\varepsilon} \cap [0,t]} \lVert \Phi_f(s-,X_{s-})B(\Delta X_s^{\otimes 2} )\rVert 
    = 0,  \notag
\end{align}
which completes the proof for the weak case.
If $Q_B(X,X)$ is the strong $B$-quadratic variation of $X$,
then the convergence of $I^{(n)}_3$ holds in the norm topology.
In this case, we obtain the norm convergence of discrete $B'$-quadratic variation
by replacing \eqref{5.2ic} with a similar norm estimate.
\end{proof}

\begin{Cor} \label{5.2j}
In addition to the assumptions in Theorem~\ref{5.2b},
assume that there is an increasing divergent sequence of positive numbers $(T_n)$
and an increasing sequence of compact convex sets $(K_n)$ such that
$(T_n,K_n)$ satisfies conditions (i)--(iii)
in Theorem~\ref{5.2b} for each $n \in \mathbb{N}$.
Then \eqref{5.2c} holds for all $t \in \mathbb{R}_{\geq 0}$.
\end{Cor}

\begin{Rem} \label{5.2k}
The assumption on $(\pi_n)$ in Corollary~\ref{5.2j} is
always satisfied whenever $\lvert \pi_n \rvert \to 0$.
If $(T_n,K_n)$ satisfies $X([0,T_n]) \subset K_n$
and $\lim_{k \to \infty} \sup_{x \in K_n} O^-_t((X,f(\cdot,x)),\pi_k) = 0$
for all $n$, then the assumption on $(\pi_n)$ in Corollary~\ref{5.2j}
is also satisfied. See Example~\ref{5.1c} and Appendix B.
\end{Rem}

To derive $C^1$-transformation formula for tensor quadratic variations,
recall the definition of a uniform crossnorm.
Let $\alpha$ be a system that defines
a reasonable crossnorm on $E \otimes F$ for each pair of Banach spaces $(E,F)$.
We say that $\alpha$ is a uniform crossnorm
if for arbitrary four Banach spaces $E_1$, $E_2$, $F_1$, $F_2$
and for all $(T,S) \in \mathcal{L}(E_1,E_2) \times \mathcal{L}(F_1,F_2)$,
the tensor product $T \otimes S$ defines a bounded operator
$E_1 \widehat{\otimes}_{\alpha} E_2 \to F_2 \widehat{\otimes}_{\alpha} F_2$
satisfying $\lVert T \otimes S \rVert \leq \lVert T \rVert \lVert S \rVert$.
In the remainder of this section, let $\alpha$ be a fixed uniform crossnorm.

\begin{Cor} \label{5.2kb}
Let $X \colon \mathbb{R}_{\geq 0} \to E$ be a 
c\`{a}dl\`{a}g path and $f \colon \mathbb{R}_{\geq 0} \to C^1_{\mathcal{K}}(E,F)$
be a {\cadlag} path such that the family
$(f(\cdot,x); x \in K)$ has uniformly finite variation for all
compact sets $K \subset E$.
Assume that there is a sequence $0 < T_0 < T_1 < \dots <T_n < \dots \to \infty$
and an increasing sequence of compact convex subsets of $E$, denoted by $(K_n)$,
such that $(T_n,K_n)$ satisfies conditions (i)--(iii):
\begin{enumerate}
    \item the image $X([0,T_n]) \subset K_n$ is included in $K$;
    \item the sequence $(\pi_n)$ approximates $(X,f(\cdot,x), D_x f(\cdot,x))$ from the left for all $x \in K_n$;
    \item the sequence $(\pi_n)$ satisfies (UC) for the family $(X,f(\cdot,x))_{x \in K_n}$.
\end{enumerate}
If $X$ has strong (resp. weak) $\alpha$-tensor quadratic variation and finite 2-variation along $(\pi_n)$,
then $f_{\cdot}(X_{\cdot})$ has the strong (resp. weak) $\alpha$-tensor quadratic variation,
given by
\begin{equation}
    {^\alpha [f(\cdot,X_{\cdot}),f(\cdot,X_{\cdot})]}_t
    =
    \int_0^t
            D_x f (s-,X_{s-})^{\otimes 2}
            \mathrm{d}{^\alpha [X,X]}^{\mathrm{c}}_s
    + \sum_{0 < s \leq t} \Delta f(s,X_s)^{\otimes 2}, \qquad t \geq 0.
\end{equation}
\end{Cor}

Next, we consider a case in which $f$ of Corollary~\ref{5.2kb}
is represented as $f(t,x) = \widetilde{f}(A_t,x)$
for some function $\widetilde{f}$ and a {\cadlag} path $A$ of finite variation.

\begin{Cor} \label{5.2l}
Let $X \in D(\mathbb{R}_{\geq 0},E)$ and $A \in FV(\mathbb{R}_{\geq 0},E)$.
Moreover, let $f \colon F \times E \to G$ be a function satisfying the following conditions:
\begin{enumerate}
    \item the restriction of $f$ to each compact subset of $F \times E$ is Lipschitz continuous;
    \item the map $x \mapsto f(a,x)$ is G\^{a}teaux differentiable and 
        $D_x f \in C_{\mathcal{K}}(F \times E,\mathcal{L}(E,G))$.
\end{enumerate}
Suppose that $(\pi_n)$ satisfies either $\lim_{n \to \infty} \lvert \pi_n \rvert = 0$
or $\lim_{n \to \infty} O^-_t((A,X),\pi_n) = 0$ for all $t \geq 0$.
If $X$ has strong (resp. weak) $\alpha$-tensor quadratic variation along $(\pi_n)$,
then $f(A,X)$ has the strong (resp. weak) $\alpha$-tensor
quadratic variation, given by
\begin{equation*}
    {^\alpha [f(A,X),f(A,X)]}_t
    =
    \int_0^t D_x f(A_{s-},X_{s-})^{\otimes 2} \mathrm{d}{^\alpha [X,X]}^{\mathrm{c}}_s
    + \sum_{0 < s \leq t} \Delta f(A,X)_s^{\otimes 2}.
\end{equation*}
\end{Cor}

Notice that Condition~(C) for $(A,X)$ is not sufficient to prove the corollary
because a large jump of $A$ is not necessarily a large jump of $t \mapsto f(A_t,x)$.

\begin{proof}
Let $g(t,x) = f(A_t,x)$.
Then, it suffices to check that
$g$ and $X$ satisfy the assumptions in Theorem~\ref{5.2b} and Corollary~\ref{5.2j}.
We see that $a \mapsto f(a,\cdot)$ is a function of the class
$C_{\mathcal{K}}(F,C^1_{\mathcal{K}}(E,G))$
by conditions~(i) and (ii), and therefore 
$t \mapsto g(t,\cdot)$ belongs to $D(\mathbb{R}_{\geq 0},C^1_{\mathcal{K}}(E,G))$.
Condition~(i) and the fact that $A$ has finite variation
shows that $(g(\cdot,x))_{x \in K}$ has uniformly finite variation
for each compact $K \subset E$.
Moreover, since the canonical bilinear map
$\mathcal{L}(E,F) \times \mathcal{L}(E,F) \to \mathcal{L}(E \widehat{\otimes}_{\alpha} E, F \widehat{\otimes}_{\alpha} F)$
is continuous, the path $t \mapsto D_x g(t,\cdot)^{\otimes 2}$
is {\cadlag} as a $C_{\mathcal{K}}(E,\mathcal{L}(E \widehat{\otimes}_{\alpha} E, F \widehat{\otimes}_{\alpha} F))$-valued path.

It remains to show that the conditions on $(\pi_n)$ hold.
If $\lvert \pi_n \rvert \to 0$,
then $(\pi_n)$ satisfies the expected assumptions by Remark~\ref{5.2k}.
Otherwise, we suppose $O_t^-(A,X;\pi_n) \to 0$ as $n \to \infty$.
Then there is a $C>0$ such that
$\sup_{x \in K} O^-_t((X,g(\cdot,x)),\pi_n) \leq C O^-_t((X,A),\pi_n)$
for each $t > 0$ and compact set $K \subset E$.
This implies the desired conditions again by Remark~\ref{5.2k}.
\end{proof}

\section{Quadratic variation of the It\^{o}--F\"{o}llmer integrals}

Let $E$ and $G$ be Banach spaces
and $X$ be a {\cadlag} path that has
tensor quadratic variation and
finite 2-variation along a sequence of partitions $(\pi_n)$.
Then, by the It\^{o}--F\"{o}llmer formula (Theorem~\ref{2e}),
\begin{align*}
    \int_0^t D_xf(X_{s-}) \mathrm{d}X_s
    & =
    f(X_t) - f(X_0)
    - \frac{1}{2} \int_0^t D_x^2 f(X_{s-})\mathrm{d}[X,X]^{\mathrm{c}}_s   \\
    & \quad - \sum_{0 < s \leq t} \{ \Delta f(X_s) - D_x f(X_{s-})\Delta X_s \}
\end{align*}
holds for a function $f \colon E \to G$ with $C^2$-smoothness in an appropriate sense.
In this section, we study the properties of the path 
$t \mapsto \int_0^t D_xf(X_{s-}) \mathrm{d}X_s$,
focusing on its quadratic variation.
In semimartingale theory, 
the stochastic integral of a suitable predictable process
with respect to a semimartingale
is again a semimartingale,
and its quadratic variation can be explicitly
computed by using the quadratic variation
of the integrator.
We aim to prove a corresponding formula
in the It\^{o}--F\"{o}llmer calculus in Banach spaces.
Such formulae have been already
given by Sondermann~\cite{Sondermann_2006}
and Schied~\cite{Schied_2014} for the finite-dimensional
and continuous case
and by Hirai~\cite{Hirai_2019} for the
finite-dimensional and {\cadlag} case.
Ananova and Cont~\cite{Ananova_Cont_2017}
gave a corresponding result,
which they call the pathwise isometry formula,
in a more general situation where the function $f$ is a path-dependent one.

\begin{Th} \label{8b}
Let $\alpha$ be a uniform crossnorm and
assume that $X$, $A$, $f$, and $(\pi_n)$ satisfy the same conditions
as those in Corollary~\ref{2fc}.
Define a $G$-valued {\cadlag} path $Y$ by
\begin{equation*}
    Y_t = \int_0^t D_xf(A_{s-},X_{s-}) \mathrm{d}X_s, \qquad t \geq 0.
\end{equation*}
Then $Y$ correspondingly has the strong or weak $\alpha$-tensor quadratic variation given by
\begin{equation} \label{8c}
    {^\alpha [Y,Y]}_t 
    = 
    \int_0^t D_x f(A_{s-},X_{s-})^{\otimes 2} \mathrm{d}{^\alpha [X,X]}_s, \qquad
    t \geq 0.
\end{equation}
\end{Th}

\begin{proof}
We show the assertion for weak tensor quadratic variations.
By Corollary~\ref{5.2l},
the path $f(A,X)$ has the weak $\alpha$-tensor quadratic variation
\begin{equation*}
    {^\alpha [f(A,X),f(A,X)]}_t
    =
    \int_0^t
        D_x f(A_{s-},X_{s-})^{\otimes 2}
        \mathrm{d}{^\alpha [X,X]}^{\mathrm{c}}_s
    + \sum_{0 < s \leq t} (\Delta f(A_s,X_s))^{\otimes 2}.
\end{equation*}
Now let 
\begin{align*}
    B_t 
    & =
    \int_0^t D_af(A_{s-},X_{s-})\mathrm{d}A^{\mathrm{c}}_s
    + \frac{1}{2} \int_0^t D_x^2 f(A_{s-},X_{s-}) \mathrm{d}[X,X]^{\mathrm{c}}_s     \\
    & \quad 
    + \sum_{0 < s \leq t} \{ \Delta f(A_s,X_s) - D_x f(A_{s-},X_{s-}) \Delta X_s \}.
\end{align*}
Then by the It\^{o} formula, we have $f(A_t,X_t) = Y_t + B_t$.
According to Corollary~{5.10} of Hirai~\cite{Hirai_2022a},
we see that $Y$ has the weak $\alpha$-tensor quadratic variation 
\begin{equation*}
    {^\alpha [Y,Y]}_t = {^\alpha [f(A,X),f(A,X)]}_t - {^\alpha [f(A,X),Y]}_t - {^\alpha [Y,f(A,X)]}_t + {^\alpha [B,B]}_t.
\end{equation*}
By direct calculations of 
$[f(A,X),Y]$, $[Y,f(A,X)]$, and $[B,B]$ 
using Corollary~{5.9} in \cite{Hirai_2022a},
we consequently obtain formula~\eqref{8c}.
\end{proof}

\begin{Cor} \label{8d}
Under the assumptions of Theorem~\ref{8b},
the c\`{a}dl\`{a}g path $f(A,X)$ admits a decomposition 
\begin{equation*}
    f(A_t,X_t) = Y_t + C_t + D_t,
\end{equation*}
where $Y$ is a c\`{a}dl\`{a}g path having the $\alpha$-tensor quadratic variation given by 
\eqref{8c}, $C$ is a continuous path of finite variation, 
and $D$ is a purely discontinuous path of finite variation. 
\end{Cor}

\appendix
\section{Remarks on the Radon--Nikodym property}

Let $\mathcal{R}$ be a ring of subsets of $\Omega$
and $\mu \colon \mathcal{R} \to E$ be a 
finitely additive vector measure.
Here, recall that the variation of $\mu$ on $A \subset \Omega$
is defined as
\begin{equation*}
    \lvert \mu \rvert (A)
    =
    \sup \left\{
        \sum_{i} \lVert \mu(A_i) \rVert\,
        \middle\vert\, 
        \text{
            $(A_i)$ is a finite disjoint family of 
            elements of $\mathcal{R}$ and 
            $\bigcup A_i \subset A$.
        }
    \right\}
\end{equation*}
Then the function $\lvert \mu \rvert \colon \mathcal{R} \to \lbrack 0,\infty \rbrack$
defines a finitely additive measure.
We say that $\mu$ has bounded variation if $\lvert \mu \rvert(\Omega) < \infty$
and finite variation if $\lvert \mu \rvert(A) < \infty$ for all $A \in \mathcal{R}$.
If, moreover, $\mu$ is countably additive on $\mathcal{R}$,
then the variation $\lvert \mu \rvert$ is also 
countably additive on $\mathcal{R}$.

We quote the following proposition
in the theory of vector integration.
See Dinculeanu~\cite[Theorem 2.29]{Dinculeanu_2000} for a proof.
\begin{Prop} \label{a1b}
Given a positive measure space $(\Omega,\mathcal{A},\mu)$
and a $f \in L^1(\mu;E)$, define
\begin{equation*}
    (f \cdot \mu)(A) = \int_A f \, \mathrm{d}\mu
\end{equation*}
for each $A \in \mathcal{A}$.
Then $f \cdot \mu$ is an $E$-valued countably additive measure
of bounded variation.
Moreover, the variation is given by 
$\lvert f \cdot \mu \rvert = \lVert f \rVert \cdot \mu$
\end{Prop}

\begin{Def} \label{a1c}
Let $(\Omega,\mathcal{A},\mu)$ be a finite positive measure space. 
A Banach space $E$ has the \emph{Radon--Nikodym property (RNP)
with respect to $(\Omega,\mathcal{A},\mu)$}
if for every $\mu$-absolutely continuous vector measure
$\nu \colon \mathcal{A} \to E$ of bounded variation,
there exists a $g \in L^1(\mu;E)$ such that 
\begin{equation*}
    \nu(A) = \int_A g(\omega) \, \mu(\mathrm{d}\omega) 
\end{equation*}
for all $A \in \mathcal{A}$.
The Banach space $E$ has the RNP
if it has the RNP with respect to
every finite positive measure space.
\end{Def}

It is known that every reflexive Banach space
and every separable dual space have the RNP
(see Diestel and Uhl~\cite{Diestel_Uhl_1977}).

\begin{Prop} \label{a1d}
Let $\mathcal{R}$ be a $\delta$-ring of subsets of $\Omega$
and let $\mu \colon \mathcal{R} \to \lbrack 0,\infty \rbrack$
be a $\sigma$-finite measure.
Suppose that the Banach space $E$ has the RNP
and $\nu \colon \mathcal{R} \to E$ is a vector measure
of finite variation.
If $\nu$ is $\mu$-absolutely continuous,
then there is a unique strongly measurable function $f \colon \Omega \to E$
such that $f 1_A \in L^1(\mu)$, and 
\begin{equation*}
    \nu(A) = \int_A f \, \mathrm{d}\mu
\end{equation*}
for all $A \in \mathcal{R}$.
\end{Prop}

\begin{proof}
First, take a $\mathcal{R}$-measurable partition $(\Omega_n)_{n \in \mathbb{N}}$
of $\Omega$ such that $\mu(\Omega_n) < \infty$ for all $n$.
Moreover, define a sequence of 
measures of bounded variation $(\nu_n)$
by $\nu_n(A) = \nu(A \cap \Omega_n)$.
Since each $\Omega_n \cap \mathcal{R}$ is a $\delta$-algebra (and hence a $\sigma$-algebra)
and $\nu_n$ is a measure of bounded variation,
there is a function $f_n \in L^1(\Omega_n,\mu\vert_{\mathcal{R} \cap \Omega_n};E)$
such that 
\begin{equation*}
    \nu_n(\Omega) = \int_{A \cap \Omega_n} f_n \,\mathrm{d}\mu, \qquad A \in \mathcal{R}.
\end{equation*}

Next, set $f(\omega) = f_n(\omega)$ for $\omega \in \Omega_n$.
Then the function $f$ is
strongly $\sigma(\mathcal{R})$-measurable. Here, note that for $A \in \mathcal{R}$,
$f$ is $\mu$-integrable on $A$.
Indeed, we have
\begin{equation*}
    \int_A \lVert f(\omega) \rVert \mathrm{d}\mu
    = 
    \sum_{n \in \mathbb{N}} \int_{A \cap \Omega_n} \lVert f_n(\omega) \rVert \mathrm{d}\mu   \\
    =
    \sum_{n \in \mathbb{N}} \vert \nu_n \rvert (A \cap \Omega_n)   \\
    =
    \sum_{n \in \mathbb{N}} \vert \nu \rvert (A \cap \Omega_n)
    =
    \lvert \nu \rvert (A)
    <
    \infty.
\end{equation*}
By the $\sigma$-additivity of $\nu$ and $\mu$ on $\mathcal{R}$,
we can justify the following calculation:
\begin{equation*}
    \int_A f(\omega) \mathrm{d}\mu
    =
    \sum_{n \in \mathbb{N}} \int_{A \cap \Omega_n} f(\omega) \mathrm{d}\mu
    =
    \sum_{n \in \mathbb{N}} \nu_n(A \cap \Omega_n)
    =
    \sum_{n \in \mathbb{N}} \nu(A \cap \Omega_n)
    =
    \nu(A).
\end{equation*}
Thus, we obtain the assertion.
\end{proof}

\section{Supplements on families of c\`{a}dl\`{a}g paths}

In this section,
we consider the problem of uniformly controlling
the oscillation of a family of c\`{a}dl\`{a}g paths
by a sequence of partitions.
To observe that, we first introduce the notion of equi-right-regularity.

\begin{Def} \label{a2.1b}
Let $E$ be a Banach space and $\mathcal{F}$ a subset of $D(\mathbb{R}_{\geq 0},E)$.
\begin{enumerate}
    \item The set $\mathcal{F}$ is \emph{equi-right-continuous}\index{equi-right continuous}
        at $t \in \mathbb{R}_{\geq 0}$
        if for every $\varepsilon > 0$
        there is a $\delta > 0$ such that
        \begin{equation*}
            \lVert f(t) - f(s) \rVert_F < \varepsilon
        \end{equation*}
        holds for all $f \in \mathcal{F}$ and $s \in [t,t+\delta \rlbrack$.
    \item The set $\mathcal{F}$ is \emph{equi-right-regular}\index{equi-right-regular}
        at $t > 0$ if it is equi-right-continuous at $t$
        and, for every $\varepsilon > 0$,
        there is a $\delta \in \lrbrack 0,t \rlbrack$ such that
        \begin{equation*}
                \lVert f(t-) - f(s) \rVert_F < \varepsilon
        \end{equation*}
        holds for all $f \in \mathcal{F}$ and $s \in \lrbrack t-\delta,t \rlbrack$.
    \item The set $\mathcal{F} \subset D(\mathbb{R}_{\geq 0},E)$ is \emph{equi-right-regular}
        if it is equi-right-regular at every $t > 0$
        and equi-right continuous at $0$.
\end{enumerate}
\end{Def}

A parametrized family $(f_{i})_{i \in I}$ of c\`{a}dl\`{a}g paths in $E$ is 
said to be equi-right-regular
if the set $\{ f_i \mid i \in I \} \subset D(\mathbb{R}_{\geq 0},E)$
is equi-right-regular.
The notion of equi-right-regularity is an analogue of 
the equicontinuity of the set of continuous functions.
A subset $\mathcal{F}$ of the space $C(\mathbb{R}_{\geq 0},E)$ is 
equi-right-regular if and only if it is equicontinuous.
We can prove an Arzel{\`a}--Ascoli-like theorem 
for the space $D(\mathbb{R}_{\geq 0},E)$ with the topology of 
uniform convergence on compact sets.
Note that like equicontinuity, equi-right-regularity can be characterized 
in the language of uniform convergence.

Let us give some examples of equi-right-regular families of paths.

\begin{Exm} \label{a2.1c}
Let $E$ and $F$ be Banach spaces.
\begin{enumerate}
    \item Every finite subset of $D(\mathbb{R}_{\geq 0},E)$ is equi-right-regular.  
    \item Let $\mathfrak{S}$ be a family of subsets of $E$ 
        and $C_{\mathfrak{S}}(E,F)$ be the set of all functions
        whose restriction to each $S \in \mathfrak{S}$ are continuous.
        We regard $C_{\mathfrak{S}}(E,F)$ as a topological space
        by the topology of uniform convergence on each member of $\mathfrak{S}$,
        i.e., the topology of $\mathfrak{S}$-convergence.
        See, for example, Bourbaki~\cite[Section X.1]{Bourbaki_1966b}
        for details about this topology.
        Then the family of paths $(f(\,\cdot\, , x); x \in S)$ is equi-right-regular
        for every $f \in D(\mathbb{R}_{\geq 0},C_{\mathfrak{S}}(E,F))$ and $S \in \mathfrak{S}$.
\end{enumerate}
\end{Exm}

If a family of c\`{a}dl\`{a}g path $\mathcal{F}$ is equi-right-regular,
its oscillation can be controlled uniformly by a partition.
The following lemma is a generalization of
Lemma~1 in Section~12 of Billingsley~\cite[122]{Billingsley_1999}
to equi-right-regular families.

\begin{Prop} \label{a2.1d}
Let $\mathcal{F}$ be an equi-right-regular subset of 
$D(\mathbb{R}_{\geq 0},E)$.
\begin{enumerate}
    \item For each $t > 0$ and $\varepsilon > 0$,
        there is a partition $\pi \in \Par([0,t])$ that satisfies
        \begin{equation*}
            O^-(f,\pi)
            \coloneqq \sup_{\lrbrack r,s \rbrack \in \pi} \omega(f, \lbrack r,s \rlbrack)
            < \varepsilon
            \qquad \text{for all $f \in \mathcal{F}$.}
        \end{equation*}
    \item The set 
        $\{ s \in [0,t] \mid
                \text{
                    $\lVert \Delta f(s) \rVert \geq C$
                    for some $f \in \mathcal{F}$
                }
            \}$
        is finite for every $t > 0$ and $C > 0$.
    \item Let $t \geq 0$ and suppose that $\lVert f(s) \rVert < C$ holds
        for all $s \in [0,t]$ and $f \in \mathcal{F}$.
        Then for every $\varepsilon > 0$,
        there exists a $\delta > 0$ such that 
        $\lVert f(s) - f(r) \rVert_E < C + \varepsilon$
        for all $f \in \mathcal{F}$
        whenever $s,r \in [0,t]$ satisfies
        $\lvert s - r \rvert < \delta$.    
\end{enumerate}
\end{Prop}

\begin{proof}
(i)
Fix $t > 0$. Define
\begin{equation*}
    \mathbb{T} =
    \{ t' \in [0,t]
        \mid
        \text{there exists a partition $\pi \in \Par[0,t']$ satisfying
        $O^-(f,\pi) < \varepsilon$ for all $f \in \mathcal{F}$}
    \}
\end{equation*}
and $t^* = \sup \mathbb{T}$.
Then it suffices to show that $t^* = t$ and $t^* \in \mathbb{T}$.

Note that $\mathbb{T} \neq \emptyset$ and hence $t^* > 0$
by the equi-right continuity of $\mathcal{F}$ at $0$.
Since $\mathcal{F}$ is equi-right-regular at $t^*$, we can
take a $\delta \in \lrbrack 0, t^* \rlbrack$ such that 
$\lVert f(s) - f(t^*-) \rVert_E < \varepsilon$
for all $f \in \mathcal{F}$
and $s \in \lrbrack t^*-\delta,t^* \rlbrack$.
Next, choose a $t' \in \lrbrack t^*-\delta, t^* \rbrack \cap \mathbb{T}$ and 
a $\pi' \in \Par([0,t'])$ so that
$O^-(f,\pi') < \varepsilon$.
Then the partition
$\pi'' = \pi' \cup \{ \lrbrack t',t^* \rbrack \} \in \Par([0,t^*])$
satisfies $O^-(f,\pi'') < \varepsilon$.
Hence $t^* \in T$.

Now assume that $t^* < t$.
Since $\mathcal{F}$ is equi-right-regular at $t^\ast$,
we can take a $\tilde{t} > t^*$ such that 
$\lVert f(t^*) - f(s) \rVert < \varepsilon$
for all $f \in \mathcal{F}$ and $s \in [t^*,\tilde{t} \rlbrack$.
Moreover, choose a $\pi \in \Par([0,t^*])$
satisfying $O^-(f,\pi) < \varepsilon$
for all $f \in \mathcal{F}$.
Then the new partition
$\tilde{\pi} = \pi \cup \{ \lrbrack t^*,\tilde{t} \rbrack \}$
satisfies $O^-(f,\tilde{\pi}) < \varepsilon$ for all $f \in \mathcal{F}$
and thus we obtain $\tilde{t} \in T$.
This contradicts the definition of $t^*$.

(ii)
Let $C > 0$ and $t > 0$.
Then, by (i), we can choose a $\pi \in \Par([0,t])$
satisfying $O^-(f,\pi) < C$ for all $f \in \mathcal{F}$.
Then we see that
\begin{equation*}
    \{
            s \in [0,t]
            \mid
            \text{
                $\lVert \Delta f(s) \rVert \geq C$
                for some $f \in \mathcal{F}$
            }
    \}
    \subset \pi^{\mathrm{p}}.
\end{equation*}
Since $\pi^{\mathrm{p}}$ is finite, by definition,
the set of jumps of $\mathcal{F}$ greater than $C$
is also finite.

(iii)
Fix $\varepsilon > 0$ arbitrarily
and choose a $\pi \in \Par([0,t])$
such that $O^-(f,\pi) < \varepsilon/2$ for all $f \in \mathcal{F}$.
Moreover, take a $\delta \in \lrbrack 0,\lvert \pi \rvert \rlbrack$.
Then, each nonempty interval $\lrbrack r,s \rbrack \subset [0,t]$ satisfying
$\lvert s -r \rvert < \delta$ contains at most one element of $\pi^{\mathrm{p}}$.
If there is no element of $\pi^{\mathrm{p}}$ in $\lrbrack r,s \rbrack$,
we have
\begin{equation*}
    \lVert f(s) - f(r) \rVert_E
    \leq O^-(f,\pi)
    < \varepsilon
\end{equation*}
for all $f \in \mathcal{F}$.
On the other hand, if 
$\lrbrack r,s \rbrack$ contains a $u \in \pi^{\mathrm{p}}$,
then
\begin{align*}
    \lVert f(s) - f(r) \rVert_E
    & \leq
    \lVert f(r) - f(u) \rVert_E
    + \lVert \Delta f(u) \rVert_E
    + \lVert f(u) - f(s) \rVert_E   \\
    & \leq
    2O^-(f,\pi) + C
    < C + \varepsilon
\end{align*}
for all $f \in \mathcal{F}$.
This completes the proof.
\end{proof}

The notion of equi-right-regularity
is useful to characterize the compactness
in the space $D(\mathbb{R}_{\geq 0},E)$
with the topology of uniform convergence on compact subsets.

\begin{Prop} \label{a2.3b}
For $\mathcal{F} \subset D(\mathbb{R}_{\geq 0},E)$, the following conditions are equivalent:
\begin{enumerate}
    \item $\mathcal{F}$ is relatively compact with respect to
        the topology of uniform convergence on compact subsets.
    \item $\mathcal{F}$ satisfies the following conditions:
        \begin{enumerate}
            \item for any $t \geq 0$, the set $\{ f(t) \mid f \in \mathcal{F} \}$
                is relatively compact in $E$;
            \item the family $\mathcal{F}$ is equi-right-regular.
        \end{enumerate}
\end{enumerate}
\end{Prop}

\begin{proof}
For convenience, let $\mathcal{O}_{\mathrm{p}}$ and
$\mathcal{O}_{\mathrm{c}}$ denote the topology of pointwise convergence
and that of uniform convergence on compact sets on $D(\mathbb{R}_{\geq 0},E)$, respectively. 
Given an $\varepsilon > 0$ and a compact interval $[0,t]$,
define the entourage $U_{\varepsilon,t}$ in $D(\mathbb{R}_{\geq 0},E)$ as
\begin{equation*}
    U_{\varepsilon,t}
    =
    \{ 
        (f,g) \in D(\mathbb{R}_{\geq 0},E) \times D(\mathbb{R}_{\geq 0},E)
        \mid
        \text{
            $\lVert f(s) - g(s) \rVert < \varepsilon$
            for all $s \in [0,t]$
        }
    \}.
\end{equation*}
Then the uniformity of uniform convergence on compact sets on $D(\mathbb{R}_{\geq 0},E)$
is generated by the family $(U(\varepsilon,t); \varepsilon,t > 0 )$.

\emph{(i)~$\implies$~(ii).}
Suppose that $\mathcal{F}$ is relatively compact in $D(\mathbb{R}_{\geq 0},E)$
with respect to $\mathcal{O}_{\mathrm{c}}$.
Because every evaluation mapping 
$\mathrm{ev}_t \colon D(\mathbb{R}_{\geq 0},E) \to E$ is continuous 
with respect to $\mathcal{O}_{\mathrm{p}}$, and hence to $\mathcal{O}_{\mathrm{c}}$,
the image $\mathrm{ev}_t(\mathcal{F})$ is relatively compact in $E$.

Next, we show that $\mathcal{F}$ is equi-right-regular.
Let $s \geq 0$ and choose a positive number $t$ with $t > s$.
Given an $\varepsilon > 0$ and $f \in \mathcal{F}$, set
\begin{equation*}
    U_{\varepsilon,t}(f)
    =
    \{ 
        g \in  D(\mathbb{R}_{\geq 0},E)
        \mid
        (f,g) \in U_{\varepsilon,t}
    \}.
\end{equation*} 
Since $\mathcal{F}$ is totally bounded,
there are finite elements $f_1,\dots,f_N \in \mathcal{F}$ such that
\begin{equation*}
    \mathcal{F} \subset \bigcup_{1 \leq i \leq N} U_{\varepsilon/3,t}(f_i).
\end{equation*}
For each $i \in \{ 1,\dots, N \}$, choose a $\delta_i \in \lrbrack 0,t-s \rlbrack$
such that
\begin{equation*}
    u \in [s,s+\delta_i \rlbrack \implies \lVert f(u) - f(s) \rVert < \frac{\varepsilon}{3}.
\end{equation*}
Next, we define $\delta = \bigwedge_{1 \leq i \leq N} \delta_i$.
For $f \in \mathcal{F}$, choose an $i \in \{ 1,\dots, N \}$
such that $f \in U_{\varepsilon/3,t}(f_i)$.
Then for all $u \in \lbrack s, s+\delta \rlbrack$,
\begin{align*}
    \lVert f(u) - f(s) \rVert
    & \leq
    \lVert f(u) - f_i(u) \lVert 
    + \lVert f_i(u) - f_i(s) \rVert 
    + \lVert f_i(s) - f(s)  \rVert   \\
    & \leq
    \frac{\varepsilon}{3} + \frac{\varepsilon}{3} + \frac{\varepsilon}{3}
    =
    \varepsilon.
\end{align*}
Similarly, if $s > 0$ we can choose a $\gamma > 0$ that fulfils
\begin{equation*}
    \lVert f(u) - f(s-) \rVert < \varepsilon
\end{equation*}
for all $f \in \mathcal{F}$ and $u \in \lrbrack s-\gamma,s \rlbrack$.
Thus, we can conclude that $\mathcal{F}$ is equi-right-regular.

\emph{(ii)~$\implies$~(i).}
Assume that $\mathcal{F}$ satisfies conditions~(a) and (b).

Take an arbitrary ultrafilter $\mathfrak{B}$ on $\mathcal{F}$.
Condition~(a) implies that
$\prod_{t \in \mathbb{R}_{\geq 0}} \mathrm{ev}_t(\mathcal{F})$
is totally bounded in $E^{\mathbb{R}_{\geq 0}}$
with respect to the uniformity of pointwise convergence\footnote{
    This coincides with the restriction of product uniformity.
}.
Therefore, the filter generated by $\mathfrak{B}$ on $E^{\mathbb{R}_{\geq 0}}$,
which is indeed an ultrafilter,
converges to an element $g$ of $E^{\mathbb{R}_{\geq 0}}$.

Now let us show that $g \in D(\mathbb{R}_{\geq 0},E)$.
Fix a $t \geq 0$ and an $\varepsilon$,
and choose a corresponding $\delta > 0$ such that
\begin{equation*}
    \lVert f(t) - f(s) \rVert < \frac{\varepsilon}{3}
\end{equation*}
holds for all $f \in \mathcal{F}$ and all $s \in [t,t+\delta \rlbrack$.
For an arbitrary $s \in [t,t+\delta \rlbrack$, 
choose a $B \in \mathfrak{B}$ satisfying both
$\mathrm{ev}_t(B) \subset U(g(t),\varepsilon/3)$
and $\mathrm{ev}_s(B) \subset U(g(s),\varepsilon/3)$.
Such a $B$ indeed exists, because $g(t)$ and 
$g(s)$ are cluster points of
$\mathrm{ev}_t(\mathfrak{B})$ and $\mathrm{ev}_s(\mathfrak{B})$,
respectively, and $\mathfrak{B}$ is a filter.
Then,
\begin{align*}
    \lVert g(t) - g(s) \rVert
    & \leq
    \lVert g(t) - f(t) \rVert + \lVert f(t) - f(s) \rVert + \lVert f(s) - g(s) \rVert \\
    & <
    \frac{\varepsilon}{3} + \frac{\varepsilon}{3} + \frac{\varepsilon}{3}
    =
    \varepsilon.
\end{align*}
Since $s$ is an arbitrary element of $[t,t+\delta \rlbrack$,
we can conclude that $g$ is right continuous at $t$.

Next, we show that $g$ has a left limit at every $t > 0$.
Let $t > 0$ and $\varepsilon > 0$,
and choose a $\gamma \in \lrbrack 0, t \rlbrack$ such that
\begin{equation*}
    \lVert f(s) - f(t-) \rVert < \frac{\varepsilon}{4}
\end{equation*}
holds for all $f \in \mathcal{F}$ and all $s \in \lrbrack t-\gamma, t \rlbrack$.
For an arbitrarily chosen $u,s \in \lrbrack t-\gamma, t \rlbrack$,
choose a $B \in \mathfrak{B}$ that meets
both $\mathrm{ev}_u(B) \subset U(g(u);\varepsilon/4)$
and $\mathrm{ev}_s(B) \subset U(g(s);\varepsilon/4)$.
Then, using a $f \in B$, we obtain the estimate
\begin{align*}
    \lVert g(u) - g(s) \rVert
    & \leq 
    \lVert g(u) - f(u) \rVert
    + \lVert f(u) - f(t-) \rVert
    + \lVert f(t-) - f(s) \rVert
    + \lVert f(s) - g(s) \rVert   \\
    & < 
    \frac{\varepsilon}{4} + \frac{\varepsilon}{4}+\frac{\varepsilon}{4} + \frac{\varepsilon}{4}
    =
    \varepsilon.
\end{align*}
This estimate, along with the completeness of $E$,
implies the existence of the left limit 
$\lim_{\substack{s \uparrow\uparrow t}} g(s)$.

It remains to show that the filter base
$\mathfrak{B}$, considered as a filter base on $D(\mathbb{R}_{\geq 0},E)$,
converges to $g$ with respect to $\mathcal{O}_{\mathrm{c}}$.
For $\varepsilon > 0$ and $t > 0$,
take a finite sequence
\begin{equation*}
    0 = t_0 < t_1 < \dots < t_N = t
\end{equation*}
such that
\begin{equation*}
    \omega(f; \lbrack t_i,t_{i+1} \rlbrack ) < \frac{\varepsilon}{3}
\end{equation*}
for all $i \in \{ 0,\dots, N-1 \}$ and $f \in \mathcal{F} \cup \{ g \}$.
Moreover, choose a $B \in \mathfrak{B}$
such that $\mathrm{ev}_{t_i}(B) \subset B(g(t_i),\varepsilon/3)$
for all $i \in \{ 0,\dots,N \}$.
Let $f \in B$ and arbitrarily choose an $s \in [0,t]$.
If $s \in \{ t_0, \dots, t_N \}$, we have
\begin{equation*}
    \lVert f(s) - g(s) \rVert < \frac{\varepsilon}{3}
\end{equation*}
by the choice of $B$.
On the other hand, if $s \in \lrbrack t_i,t_{i+1} \rlbrack$
for some $i \in \{ 0,\dots, N-1 \}$, we see that
\begin{equation*}
    \lVert f(s) - g(s) \rVert 
\leq
    \lVert f(s) - f(t_i) \rVert + \lVert f(t_i) - g(t_i) \rVert + \lVert g(t_i) - g(s) \rVert
< \varepsilon.
\end{equation*}
Therefore, $f \in U_{\varepsilon,t}(g)$.
The arbitrariness of $f$ leads to $B \subset U_{\varepsilon,t}(g)$.
Thus, the filter base $\mathfrak{B}$ on $D(\mathbb{R}_{\geq 0},E)$
converges to $g$.
\end{proof}

\begin{Prop} \label{a2.3c}
Suppose that $\mathcal{F} \subset D(\mathbb{R}_{\geq 0},E)$
is relatively compact with respect to the topology of
uniform convergence on compact sets.
If a sequence $(\pi_n)$ of partitions of $\mathbb{R}_{\geq 0}$
satisfies (UC1), (UC2), and (C3) for $\mathcal{F}$,
then it satisfies (UC) for $\mathcal{F}$.
\end{Prop}

\begin{proof}
Given an arbitrary $t \geq 0$ and a $\delta > 0$, 
choose finite elements $f_1,\dots, f_N$ of $\mathcal{F}$ such that
\begin{equation*}
    \mathcal{F} \subset \bigcup_{1 \leq i \leq N} U_{t,\delta/5}(f_i).
\end{equation*}
By assumption, we can choose an $\varepsilon_0$ such that
\begin{equation*}
    \sup_{\varepsilon < \varepsilon_0} \varlimsup_{n \to \infty}
        O_t^+(f_i-J_{\varepsilon}(f_i),\pi_n) < \frac{\delta}{5} 
\end{equation*}
for all $1 \leq i \leq N$.
Moreover, for $\varepsilon < \varepsilon_0$,
choose an $N_{\varepsilon} \in \mathbb{N}$ such that 
\begin{itemize}
    \item $[0,t] \cap D_{\varepsilon}(\mathcal{F}) \cap I$ contains at most one element for all $I \in \pi_n$, and
    \item $O_t^+(f_i-J_{\varepsilon}(f_i),\pi_n) < \delta/5$ for $i \in \{ 1,\dots, N \}$
\end{itemize}
whenever $n \geq N_{\varepsilon}$.
Let $n \geq N_{\varepsilon}$,
and arbitrarily choose $I \in \pi_n$ and $f \in \mathcal{F}$.
Then for any $u,v \in I$, we have
\begin{align*}
    &
    \lVert (f-J_{\varepsilon}(f))(u) - (f-J_{\varepsilon}(f))(v) \rVert   \\
    & \qquad \leq 
    O^+_t(f_i - J_{\varepsilon}(f_i)) + 2 \sup_{s \in [0,t]} \lVert f_i(s) - f(s) \rVert + \sum_{s \in I \cap D_{\varepsilon}} \left\{ \lVert \Delta f(s) - \Delta f_i(s) \rVert \right\}  \\
    & \qquad \leq
    O^+_t(f_i - J_{\varepsilon}(f_i)) + 4 \sup_{s \in [0,t]} \lVert f_i(s) - f(s) \rVert   \\
    & \qquad \leq 
    \delta.
\end{align*}
Therefore, we can see that
\begin{equation*}
    \varlimsup_{n \to \infty} \sup_{f \in \mathcal{F}} O_{t}^+(f-J_{\varepsilon}(f);\pi_n) \leq \delta
\end{equation*}
holds for all $\varepsilon \leq \varepsilon_0$.
This shows that
\begin{equation*}
    \varlimsup_{\varepsilon \downarrow \downarrow 0}
        \varlimsup_{n \to \infty} \sup_{f \in \mathcal{F}} 
        O_{t}^+(f-J_{\varepsilon}(f);\pi_n)
    = 0,
\end{equation*}
which is the assertion of the proposition.
\end{proof}

\subsection*{Acknowledgements}
The author thanks his supervisor Professor Jun Sekine for his helpful support and encouragement.
The author also thanks Professor Masaaki Fukasawa and Professor Masanori Hino for their helpful comments and discussion. 
This work was partially supported by
JSPS KAKENHI Grant Number JP19H00643.

\printbibliography[heading=bibintoc]

Graduate School of Science,
Kyoto University, Kyoto, Japan

\emph{E-mail}: hirai.yuki.7j@kyoto-u.ac.jp
\end{document}